\documentclass[11pt,letter]{article}
\usepackage[pagewise,mathlines]{lineno}
\usepackage{amsfonts,amssymb,enumitem}
\usepackage{amsmath,amsthm}
\usepackage[T1]{fontenc}
\usepackage{lmodern}
\usepackage[utf8]{inputenc}
\usepackage{subcaption}
\usepackage{microtype}
\usepackage{bbm,hyperref}
\hypersetup{
	colorlinks=true,
    linkcolor={red!50!black},
    citecolor={blue!50!black},
    urlcolor={blue!80!black},
	bookmarksopen=true,
	bookmarksnumbered,
	bookmarksopenlevel=2,
	bookmarksdepth=3
}
\usepackage[capitalise]{cleveref}
\Crefname{condition}{Condition}{Conditions}
\usepackage{mathtools}

\DeclarePairedDelimiter\floor{\lfloor}{\rfloor}
\usepackage{graphicx}
\usepackage{epsf}
\usepackage[margin=1in]{geometry}
\usepackage{epsfig}
\usepackage{graphics}
\usepackage{color}
\usepackage[normalem]{ulem}
\usepackage[defaultlines=3,all]{nowidow}
\usepackage{todonotes}

\usepackage{tikz}
\usetikzlibrary{positioning,shapes,calc,hobby,decorations.pathmorphing,decorations.pathreplacing}

\usepackage{titling}

\newcommand{\Bag}{X}

\usepackage{array}
\newcolumntype{L}[1]{>{\raggedright\let\newline\\\arraybackslash\hspace{0pt}}m{#1}}
\newcolumntype{C}[1]{>{\centering\let\newline\\\arraybackslash\hspace{0pt}}m{#1}}
\newcolumntype{R}[1]{>{\raggedleft\let\newline\\\arraybackslash\hspace{0pt}}m{#1}}
\usepackage{boxedminipage}
\newcommand{\problemdef}[4][XXXEMPTYLABELXXX]{
	\begin{center}
		\begin{boxedminipage}{\textwidth}
			\textsc{{#2}}\ifthenelse{\equal{#1}{XXXEMPTYLABELXXX}}{}{\label{#1}}\\[2pt]
		    \renewcommand{\arrayrulewidth}{0pt}
			\begin{tabular}{@{\hspace{0.007\textwidth}}r@{\hspace{0.007\textwidth}}p{0.87\textwidth}@{\hspace{0.007\textwidth}}}
				\textit{Input:}  & {#3}\\
				\textit{Output:}  & {#4}
			\end{tabular}
		\end{boxedminipage}
	\end{center}
}

\usepackage{etoolbox} %

\newcommand*\linenomathpatch[1]{%
  \cspreto{#1}{\linenomath}%
  \cspreto{#1*}{\linenomath}%
  \csappto{end#1}{\endlinenomath}%
  \csappto{end#1*}{\endlinenomath}%
}
\newcommand*\linenomathpatchAMS[1]{%
  \cspreto{#1}{\linenomathAMS}%
  \cspreto{#1*}{\linenomathAMS}%
  \csappto{end#1}{\endlinenomath}%
  \csappto{end#1*}{\endlinenomath}%
}

\expandafter\ifx\linenomath\linenomathWithnumbers
  \let\linenomathAMS\linenomathWithnumbers
  \patchcmd\linenomathAMS{\advance\postdisplaypenalty\linenopenalty}{}{}{}
\else
  \let\linenomathAMS\linenomathNonumbers
\fi

\linenomathpatch{equation}
\linenomathpatchAMS{gather}
\linenomathpatchAMS{multline}
\linenomathpatchAMS{align}
\linenomathpatchAMS{alignat}
\linenomathpatchAMS{flalign}
\usepackage[sort,compress,numbers]{natbib}
\usepackage{doi}

\makeatletter
\pretocmd{\NAT@citexnum}{\@ifnum{\NAT@ctype>\z@}{\let\NAT@hyper@\relax}{}}{}{}
\makeatother
\newenvironment{subproof}[1][\proofname]{%

  \begin{proof}[#1]%
}{%
  \end{proof}%
}

\newcommand{\tw}{\mathrm{tw}}

\newcommand{\mms}{\mathrm{mms}}
\newcommand{\tin}{\mathsf{tree} \textnormal{-} \alpha}

\newcommand{\FF}{\mathcal{F}}
\newcommand{\HH}{\mathcal{H}}
\newcommand{\G}{\mathcal{G}}

\renewcommand{\P}{\textsf{P}}
\newcommand{\NP}{\textsf{NP}}

\renewcommand{\O}{\mathcal{O}}

\newtheorem{theorem}{Theorem}[section]
\newtheorem{claim}[theorem]{Claim}
\newtheorem{corollary}[theorem]{Corollary}
\newtheorem{proposition}[theorem]{Proposition}
\newtheorem{lemma}[theorem]{Lemma}
\newtheorem{observation}[theorem]{Observation}

\theoremstyle{definition}
\newtheorem{remark}[theorem]{Remark}
\newtheorem{example}[theorem]{Example}
\newtheorem{definition}[theorem]{Definition}
\newtheorem{question}[theorem]{Question}
\newtheorem{conjecture}[theorem]{Conjecture}

\newtheorem{numbered-claim}{Claim}
\crefname{numbered-claim}{Claim}{Claims}
\Crefname{numbered-claim}{Claim}{Claims}
\crefname{question}{Question}{Questions}
\Crefname{question}{Question}{Questions}
\crefname{property}{property}{properties}
\Crefname{property}{Property}{Properties}
\crefformat{property}{property~(#2#1#3)}
\Crefformat{property}{Property~(#2#1#3)}
\crefrangeformat{property}{properties~(#3#1#4) to~(#5#2#6)}
\crefmultiformat{property}{properties~(#2#1#3)}{ and~(#2#1#3)}{, (#2#1#3)}{ and~(#2#1#3)}

\title{Treewidth versus clique number. III. Tree-independence number of graphs with a forbidden structure}

\preauthor{}
\author{
\begin{center}
Cl\'{e}ment Dallard\textsuperscript{1}, Martin Milani{\v c}\textsuperscript{2,3}, Kenny \v{S}torgel\textsuperscript{2,4}\\[10pt]
{\small \textsuperscript{1} LIP, École Normale Supérieure de Lyon, France}\\
{\small \textsuperscript{2} FAMNIT, University of Primorska, Koper, Slovenia}\\
{\small \textsuperscript{3} IAM, University of Primorska, Koper, Slovenia}\\
{\small \textsuperscript{4} Faculty of Information Studies in Novo mesto, Slovenia}\\[5pt]
{\footnotesize%
\url{clement.dallard@ens-lyon.fr}\quad
\url{martin.milanic@upr.si}\quad
\url{kennystorgel.research@gmail.com}}
\end{center}
}
\postauthor{}
\predate{}
\date{}
\postdate{}

\begin{document}
\maketitle
\begin{abstract}
We continue the study of $(\tw,\omega)$-bounded graph classes, that is, hereditary graph classes in which the treewidth can only be large due to the presence of a large clique, with the goal of understanding the extent to which this property has useful algorithmic implications for the Maximum Independent Set and related problems.

In the previous paper of the series [Dallard, Milani\v{c}, and \v{S}torgel, Treewidth versus clique number. {II}. Tree-independence number], we introduced the \emph{tree-independence number}, a min-max graph invariant related to tree decompositions.
Bounded tree-independence number implies both \hbox{$(\tw,\omega)$-boundedness} and the existence of a polynomial-time algorithm for the Maximum Weight Independent Packing problem, provided that the input graph is given together with a tree decomposition with bounded independence number.
In particular, this implies polynomial-time solvability of the Maximum Weight Independent Set problem.

In this paper, we consider six graph containment relations---the subgraph, topological minor, and minor relations, as well as their induced variants---and for each of them characterize the graphs $H$ for which any graph excluding $H$ with respect to the relation admits a tree decomposition with bounded independence number.
The induced minor relation is of particular interest: we show that excluding either a $K_5$ minus an edge or the $4$-wheel implies the existence of a tree decomposition in which every bag is a clique plus at most $3$ vertices, while excluding a complete bipartite graph $K_{2,q}$ implies the existence of a tree decomposition with independence number at most $2(q-1)$.

These results are obtained using a variety of tools, including $\ell$-refined tree decompositions, SPQR trees, and potential maximal cliques, and actually show that in the bounded cases identified in this work, one can also compute tree decompositions with bounded independence number efficiently.
Applying the algorithmic framework provided by the previous paper in the series leads to polynomial-time algorithms for the Maximum Weight Independent Set problem in an infinite family of graph classes, each of which properly contains the class of chordal graphs.
In particular, these results apply to the class of $1$-perfectly orientable graphs, answering a question of Beisegel, Chudnovsky, Gurvich, Milani\v c, and Servatius from 2019.
\end{abstract}

\clearpage
\tableofcontents
\clearpage

\section{Introduction}

\subsection{Background}

Treewidth is an important and well-studied graph invariant, with both structural and algorithmic applications.
The presence of a large clique implies large treewidth and it is an interesting question for which graph classes this sufficient condition is also necessary.
A graph class $\G$ is said to be \emph{$(\tw,\omega)$-bounded} if it admits a  \emph{$(\tw,\omega)$-binding function}, that is, a function $f$ such that the treewidth of any graph $G\in \G$ is at most $f(\omega(G))$, where $\omega(G)$ is the clique number of $G$, and the same holds for all induced subgraphs of $G$.
The chromatic number of any graph is bounded from above by its treewidth plus one, and hence every $(\tw,\omega)$-bounded graph class is $\chi$-bounded, that is, the chromatic number of the graphs in the class and all their induced subgraphs is bounded from above by some function of the clique number (see, e.g.,~\cite{MR951359}).
Similarly as $\chi$-boundedness generalizes perfection, $(\tw,\omega)$-boundedness generalizes chordality.
Besides the class of chordal graphs, further examples of $(\tw,\omega)$-bounded graph classes include the class of circular-arc graphs, as well as several more general families of graph classes studied in the literature (see~\cite{MR1090614,MR1642971,MR1172354,DBLP:journals/endm/ChaplickZ17,MR4249058,MR3746153,MR4141534,MR4332111,MR1852483}).
This is the third paper of a series of papers on $(\tw,\omega)$-bounded graph classes, initiated in~\cite{dallard2021treewidth} (see also~\cite{DMS-WG2020}) and continued in~\cite{dallard2022firstpaper}.

It is known that $(\tw,\omega)$-bounded graph classes enjoy good algorithmic properties related to clique and coloring problems, in some cases under the mere assumption that the class admits a computable $(\tw,\omega)$-binding function (see~\cite{DBLP:journals/endm/ChaplickZ17,MR4332111,DMS-WG2020,dallard2021treewidth,dallard2022firstpaper}).
This motivates the question of whether $(\tw,\omega)$-boundedness has any other algorithmic implications, in particular, for problems related to independent sets.
A partial answer to this question was given in the second paper of the series~\cite{dallard2022firstpaper} where we identified a sufficient condition for $(\tw,\omega)$-bounded graph classes to admit a polynomial-time algorithm for the \textsc{Max Weight Independent Packing} problem and, as a consequence, for the weighted variants of the \textsc{Independent Set} and \textsc{Induced Matching} problems. (We refer to \cref{sec:preliminaries} for precise definitions.)
These results were obtained using the notion of tree-independence number, a newly introduced graph invariant related to tree decompositions.

The \emph{independence number} of a tree decomposition of a graph $G$ is defined as the maximum independence number of a subgraph of $G$ induced by a bag of the decomposition.
The \emph{tree-independence number} of a graph $G$
is the minimum independence number over all tree decompositions of $G$.
It follows from Ramsey's theorem that every graph class with bounded tree-independence number is $(\tw,\omega)$-bounded, even with a polynomial binding function; in particular, every graph class with bounded tree-independence number is polynomially $\chi$-bounded.
It was shown in~\cite{dallard2022firstpaper} that the \textsc{Max Weight Independent Packing} problem can be solved in polynomial time provided that the input graph is given along with a tree decomposition with bounded independence number.
Thus, boundedness of the tree-independence number is an algorithmically useful refinement of \hbox{$(\tw,\omega)$-boundedness} that still captures a wide variety of families of graph classes, including graph classes of bounded treewidth, graph classes of bounded independence number, intersection graphs of connected subgraphs of graphs with bounded treewidth, and graphs in which all minimal separators are of bounded size (see~\cite{dallard2022firstpaper} for details).

\subsection{Our results}

In this paper, we identify new graph classes with bounded tree-independence number and show that in each of these cases the general algorithmic result from~\cite{dallard2022firstpaper} can be applied.

\subsubsection{Bounding the tree-independence number}

We consider six graph containment relations---the subgraph, topological minor, and minor relations, as well as their induced variants---and for each of them characterize the graphs $H$ for which any graph excluding $H$ with respect to the relation admits a tree decomposition with bounded independence number.
These results build on and refine the analogous characterizations for $(\tw,\omega)$-boundedness from~\cite{dallard2021treewidth}, and show that in all these cases, $(\tw,\omega)$-boundedness is actually equivalent to bounded tree-independence number.

Instead of stating the exact characterizations here (they can be found in \cref{dichotomy tin}), let us summarize their main features.
Because of the Grid-Minor Theorem, for any graph class closed under subgraphs (and hence also for any graph class closed under topological minors or minors), bounded tree-independence number is equivalent to bounded treewidth.
When a single graph is excluded as an induced subgraph or as  an induced topological minor, bounded tree-independence number is equivalent to bounded independence number, except for a small number of well-structured graph classes, such as disjoint unions of complete graphs, chordal graphs, or block-cactus graphs.
The most interesting case is when a single graph is excluded as an induced minor.
Given two graphs $G$ and $H$, we say that $G$ is \emph{$H$-induced-minor-free} if it is not possible to obtain from $G$ a graph isomorphic to $H$ by a sequence of vertex deletions and edge contractions.
We show that the class of $H$-induced-minor-free graphs has bounded tree-independence number if and only if $H$ is an induced minor of $W_4$ (a wheel with four spokes, that is, the graph obtained from the $4$-vertex cycle by adding to it a universal vertex), of $K_5^-$ (the complete $5$-vertex graph minus an edge), or of some complete bipartite graph $K_{2,q}$.
In fact, we not only prove that these graph classes have bounded tree-independence number, but also show that, given a graph in such a class, we can compute a tree decomposition with bounded independence number in polynomial time.

As a key ingredient of our approach, we develop a generic framework for graph classes closed under taking induced topological minors that allows us, roughly speaking, to reduce the problem of computing a tree decomposition with bounded independence number satisfying some additional requirements to the case of $3$-connected graphs, in linear time (see \cref{thm:sufficient}).
We then apply our framework to the classes of $K_5^-$- and $W_4$-induced-minor-free graphs, for which we obtain a particularly good structural understanding: any such graph admits a tree decomposition in which each bag is a clique plus at most three vertices, resembling a well-known characterization of chordal graphs (where each bag is a clique).
This property gives rise to conceptually simpler proofs of $(\tw,\omega)$-boundedness of these two graph classes compared to the proofs from~\cite{DMS-WG2020,dallard2021treewidth}, which rely on graph minors theory; it also implies that the tree-independence number is bounded by~$4$ and is likely to have further algorithmic applications besides the polynomial-time solvability of the \textsc{Max Weight Independent Packing} problem.
Furthermore, we claimed in~\cite{dallard2021treewidth} that all the identified $(\tw,\omega)$-bounded graph classes have a polynomial binding function, but left out the proofs for the case of $W_4$- and $K_5^-$-induced-minor-free graphs, explaining that we would prove these results in a future publication.
This ``future publication'' refers to the present work, where the aforementioned property implies the existence of a linear $(\tw,\omega)$-binding function for these two classes.
In particular, the classes of $W_4$- and $K_5^-$-induced-minor-free graphs are linearly $\chi$-bounded.
Prior to this work, it was not even known whether these classes are polynomially $\chi$-bounded.\footnote{Esperet asked whether every $\chi$-bounded graph class admits a polynomial $\chi$-binding function (see~\cite{esperet2017graph}).
Esperet's question was recently answered in the negative by Bria\'nski, Davies, and Walczak~\cite{brianski2022separating}.}

Let us note that the class of $W_4$-induced-minor-free graphs generalizes not only the class of chordal graphs but, more generally, also the class of universally signable graphs studied by Conforti, Cornu\'{e}jols, Kapoor, Ajai, and Vu\v{s}kovi\'{c}~\cite{MR1466576} (as can be seen using~\cite[Corollary 4.2]{MR1466576}).
Furthermore, the class of graphs excluding both $W_4$ and $K_5^-$ as an induced minor was studied in 2015 by Lewchalermvongs in his Ph.D.\ thesis~\cite{Lewchalermvongs}; this class of graphs generalizes the class of graphs of separability at most $2$ studied by Cicalese and Milani\v{c}~\cite{MR2901082}.

The classes of $K_{2,q}$-induced-minor-free graphs, for all $q \ge 3$, form an infinite family of graph classes that strictly generalize the class of chordal graphs and that have bounded tree-independence number.
It is an easy observation that a graph is $K_{2,q}$-induced-minor-free if and only if no independent set of size $q$ is contained in a minimal separator.
Hence, by showing that the tree-independence number is bounded in such classes, we generalize the fact that the classes of graphs with minimal separators of bounded size are $(\tw,\omega)$-bounded~\cite{MR1852483}.

\subsubsection{Algorithmic implications}

Our results generalize a number of results in the literature and answer several questions from the literature.
We now discuss these connections.

Most notably, we identify new graph classes where the \textsc{Max Weight Independent Set} problem is solvable in polynomial time.
This problem is one of the most studied graph optimization problems, and is known to be (strongly) \NP-hard in general~\cite{MR0378476} and \NP-hard to approximate on $n$-vertex graphs to within a factor of $n^{1-\epsilon}$ for every $\epsilon>0$~\cite{MR2403018}.
Several recent results consider the problem in restricted graph classes~\cite{MR4117301,MR4141321,MR3763297,MR3909546,MR4262549,MR4262487,MR4232071,DBLP:conf/sosa/PilipczukPR21,10.1145/3406325.3451034}.

Before explaining the implications of our results and the results from~\cite{dallard2022firstpaper} for the \textsc{Max Weight Independent Set} problem, let us put them in context with related work.
While the complexity of the \textsc{Max Weight Independent Set} problem was traditionally studied in \emph{hereditary} graph classes, that is, graph classes closed under induced subgraphs, many of the graph classes for which the complexity of the problem has been determined are also closed under the induced topological minor, induced minor, or even the minor relation.
For example:
\begin{itemize}
\item The \textsc{Max Weight Independent Set} problem is \NP-hard in the class of planar graphs~\cite{MR411240}, which is exactly the class of graphs with no $K_5$ and $K_{3,3}$ as a minor.
\item The \textsc{Max Weight Independent Set} problem is solvable in polynomial time in the class of graphs containing no induced $6$-vertex path~\cite{MR3909546}.
For every positive integer $p$, any class of graphs containing no induced $p$-vertex path is closed under induced minors  (in fact, it is exactly the class of graphs excluding the $p$-vertex path as an induced minor).
\item The \textsc{Max Weight Independent Set} problem is solvable in polynomial time in the class of graphs containing no induced cycles of length at least $5$~\cite{MR4262549}.
For every positive integer $p$, any class of graphs containing no induced cycles of length at least $p$ is also closed under induced minors (in fact, it is exactly the class of graphs excluding the $p$-vertex cycle as an induced minor).
\item The \textsc{Max Weight Independent Set} problem is solvable in polynomial time in the class of wheel-free graphs excluding all induced subdivisions of the $K_4$~\cite{MR3757557}.
This class of graphs is closed under induced topological minors.
\end{itemize}

These results motivate the study of the complexity of the \textsc{Max Weight Independent Set} problem in classes of graphs that are closed under minors, induced minors, or induced topological minors.
The simplest way to define such a class is to exclude a single graph $H$ with respect to the considered relation.
If a non-planar graph $H$ is excluded, then the \textsc{Max Weight Independent Set} problem is \NP-hard in the class of $H$-minor-free graphs, since this class contains the class of all planar graphs and the \textsc{Max Weight Independent Set} problem is known to be \NP-hard in the class of planar graphs~\cite{MR411240}.
If a planar graph $H$ is excluded as a minor, then the treewidth is bounded and the \textsc{Max Weight Independent Set} problem is solvable in linear time.
Thus, the interesting cases are when a planar graph $H$ is excluded with respect to the induced minor or the induced topological minor relation.
The resulting graph classes are equivalent for the two relations if $H$ is a path or a cycle.
The case when $H$ is the $k$-vertex path $P_k$ is polynomial-time solvable for $k\le 6$~\cite{MR3909546} and open for all $k\ge 7$; however, quasi-polynomial-time algorithms are known for every fixed $k$, as shown by Gartland and Lokshtanov~\cite{MR4232071} and Pilipczuk, Pilipczuk, and Rz\k{a}\.{z}ewski~\cite{DBLP:conf/sosa/PilipczukPR21}.
These results were recently strengthened by Gartland, Lokshtanov, Pilipczuk, Pilipczuk, and Rza\.{z}ewski~\cite{10.1145/3406325.3451034} to the case when $H$ is the $k$-vertex cycle, which is polynomial-time solvable for $k\le 5$~\cite{MR4262549} and open for all $k\ge 6$.

\medskip

Recall that for all the graph classes with bounded tree-independence number considered in this work we are able to \emph{efficiently} compute a tree decomposition with bounded independence number.
Combined with the results from~\cite{dallard2022firstpaper}, our results lead to a new infinite family of graph classes in which the \textsc{Max Weight Independent Set} problem is solvable in polynomial time.
This result is summarized in the following theorem, which combines \cref{MWIS for W4-im-free graphs,MWIS for K5minus-im-free graphs,cor:MWIS-K2q}.

\begin{theorem}\label{theorem-main}
The \textsc{Max Weight Independent Set} problem is solvable in polynomial time in the class of $H$-induced-minor-free graphs whenever $H$ is either $W_4$, $K_5^-$, or a complete bipartite graph $K_{2,q}$ for some positive integer $q$.
\end{theorem}

In~\cite{dallard2021treewidth}, we showed that the \textsc{Max Weight Independent Set} problem can be solved in polynomial time in all the identified $(\tw,\omega)$-bounded graph classes, except possibly for the case when $H$ is excluded with respect to the induced minor relation and $H$ is isomorphic to either $W_4$, $K_5^-$, or $K_{2,q}$ where $q\ge 3$.
In particular, we established polynomial-time solvability of the \textsc{Max Weight Independent Set} problem in the classes of $H$-induced-minor-free graphs when $H$ is either the diamond (that is, $K_4$ minus an edge) or the complete bipartite graph $K_{1,q}$ for some positive integer~$q$.
The results given by \cref{theorem-main} generalize these results and establish polynomial-time solvability of the \textsc{Max Weight Independent Set} problem for all the cases left open in~\cite{dallard2021treewidth}.

\Cref{theorem-main} provides an infinite family of planar graphs $H$ such that excluding $H$ as an induced minor results in polynomial-time solvability of the \textsc{Max Weight Independent Set} problem.
To further appreciate the result of \cref{theorem-main}, note that the only graph class containing all classes of \hbox{$K_{2,q}$-induced-minor-free} graphs is the class of all graphs.
Furthermore, unless $H\in \{K_{2,1},K_{5}^-\}$, each of the graph classes listed in \cref{theorem-main} generalizes the class of chordal graphs---for which a polynomial-time algorithm for \textsc{Max Weight Independent Set} was given in 1976 by Frank~\cite{MR0392683}---and, with the exception of the case $H = K_{2,2}$, all these generalizations are proper.
Additionally, the fact that \textsc{Max Weight Independent Set} is solvable in polynomial time in the class of $K_{2,3}$-induced-minor-free graphs implies a polynomial-time solution for \textsc{Max Weight Independent Set} in the class of $1$-perfectly orientable graphs, which is a common generalization of the classes of chordal graphs and circular-arc graphs.
This answers a question of Beisegel, Chudnovsky, Gurvich, Milani\v{c}, and Servatius posed in~\cite{MR3992956,MR4357319}.
The analogous question for \textsc{$k$-Coloring} was answered in our previous works~\cite{DMS-WG2020,dallard2021treewidth}.

We in fact obtain results analogous to \cref{theorem-main} but for the
\textsc{Max Weight Independent Packing} problem (\cref{MWIHP for W4-im-free graphs,MWIHP for K5minus-im-free graphs,MWIHP-K2q}).
These results generalize the polynomial-time solvability of the \textsc{Max Induced Matching} problem in the class of chordal graphs~\cite{MR1011265}, of the \textsc{$k$-Separator} problem (for fixed $k$) in the class of interval graphs~\cite{MR3296270}, and of the \textsc{Independent $\mathcal{F}$-Packing} problem in the class of chordal graphs~\cite{MR2190818}.

\subsection{Overview of our methodology}

We develop two significantly different general approaches for establishing boundedness of the tree-independence number for a given graph class, and for computing tree decompositions of bounded independence number for graphs in the class.
Given a graph class $\mathcal G$, we say that (boundedness of the) tree-independence number is \emph{efficiently witnessed} in $\mathcal{G}$ if there exists a polynomial-time algorithm to compute a tree decomposition with bounded independence number for every graph in~$\mathcal G$.

\bigskip
\noindent{\bf First approach: removing small cutsets.}
We make use of two known and linear-time computable decompositions of graphs into biconnected and triconnected components: the decompositions based on the block-cutpoint trees and the SPQR trees, respectively.
Using these decompositions, we can reduce the problem to connected graphs with no cutsets of size one and two, respectively.

The fact that block-cutpoint trees are useful for the tree-independence number follows from a more general observation that the tree-independence number is ``well-behaved'' with respect to clique cutsets, in the sense that it is bounded in any graph class closed under induced subgraphs in which graphs \emph{without clique cutsets} have bounded tree-independence number (see \cite[Proposition 3.10]{dallard2022firstpaper} and its refinement, \cite[Proposition 4.6]{dallard2022firstpaper}, stated in this work as~\cref{reduction-to-atoms}).
Indeed, the same approach shows that the tree-independence number is bounded in any graph class closed under induced subgraphs in which graphs \emph{without cut vertices} have bounded tree-independence number.
Furthermore, both reductions are algorithmically efficient.
When the graph is decomposed along a clique cutset into smaller subgraphs, if a tree decomposition with bounded independence number is provided for each subgraph, then one can combine these tree decompositions into a tree decomposition with bounded tree independence number for the entire graph.
In the case of block-cutpoint trees, the whole procedure can be implemented in linear time (\cref{cor:reduction-to-blocks}).

Getting rid of cutsets of size two using SPQR trees can also be performed in linear time but is significantly more involved.
Besides, it cannot always be applied, since it is not necessarily the case that bounded tree-independence number of the ``triconnected components'' of the graph implies boundedness of the tree-independence number of the whole graph (see \cref{example subdivided Kn}).
Additional technical conditions need to be met.
We postpone the precise statement of the result to \cref{sec:sufficient} (see \cref{thm:sufficient}), but let us remark that the result can be applied within graph classes closed under induced topological minors.

We then use this combined decomposition approach and a classical result due to Tutte~\cite{MR0140094} that any $3$-connected graph with at least $5$ vertices contains an edge whose contraction results in a $3$-connected graph to show that the tree-independence number of any graph that does not contain an induced minor isomorphic to $W_4$ or to $K_5^-$ is bounded by $4$ (see \cref{W_4-tree-independence-number,K_5-tree-independence-number}, respectively).
Furthermore, we give a linear-time algorithm to compute a tree decomposition with independence number at most $4$ of any such graph.
In the case of $W_4$-induced-minor-free graphs, particular tree decompositions of chordal graphs known as \emph{clique trees} (see~\cite{MR1320296}) also play an important role, along with moplex elimination orderings, and known linear-time algorithms to compute them.

\bigskip
\noindent{\bf Second approach: potential maximal cliques and minimal triangulations.}
The second approach relies on the notions of \emph{potential maximal cliques} and \emph{minimal triangulations}.
Roughly speaking, a minimal triangulation of a graph $G$ is a chordal graph obtained from $G$ by adding to it an inclusion-minimal set of edges, and a potential maximal clique is a set of vertices that forms a maximal clique in some minimal triangulation (see, e.g.,~\cite{MR2204109}).
By combining known results on minimal triangulations and chordal graphs, one can compute in polynomial time a tree decomposition of a given graph $G$ in which each bag is a potential maximal clique.
It follows that the tree-independence number is bounded and efficiently witnessed in any class of graphs in which there is a bound on the maximum size of an independent set contained in a potential maximal clique.

To apply this observation to particular graph classes, we connect it with another classical concept in algorithmic graph theory, that of a \emph{minimal separator}.
A minimal separator in a graph is an inclusion-minimal set of vertices separating some fixed non-adjacent vertex pair.
Building on the work of Bouchitt\'{e} and Todinca~\cite{MR1896345}, we show that in every graph, every potential maximal clique is either a clique or is contained in the union of two minimal separators (\cref{thm:pmcs-covering}).
Thus, the tree-independence number is bounded and efficiently witnessed in any class of graphs for which there is a bound on the maximum size of an independent set contained in a minimal separator.
This last condition is easily seen to be equivalent to the requirement that the graph is $K_{2,q}$-induced-minor-free for some fixed integer $q$ and leads to polynomial-time algorithms for computing tree decompositions with bounded independence number in any such graph class.

\subsection{Organization of the paper}

In \cref{sec:preliminaries}, we provide the necessary preliminaries.
Then, in \cref{sec:K2q-induced-minor-free}, we prove that for every positive integer $q$, the class of $K_{2,q}$-induced-minor-free graphs has bounded and efficiently witnessed tree-independence number.
A polynomial-time algorithm for the \textsc{Max Weight Independent Set} problem in the class of $1$-perfectly orientable graphs is also derived in that section.
In \cref{sec:sufficient}, we develop a sufficient condition for bounded and efficiently witnessed tree-independence number for graph classes closed under induced topological minors.
This general result is then applied to the specific cases of $W_4$-induced-minor-free graphs and
$K_5^-$-induced-minor-free graphs in \cref{W4-induced-minor-free,K5-induced-minor-free}, respectively.
In \cref{sec:dichotomies}, we consider the six aforementioned graph containment relations and for each of them characterize the graphs $H$ for which the class of graphs excluding $H$ has bounded tree-independence number.
We conclude the paper in \cref{sec:open-questions}.

\section{Preliminaries}\label{sec:preliminaries}

\subsection{General notations and definitions}\label{sec:preliminaries-general}

Throughout the paper, we denote by $\mathbb{Z}_+$ the set of nonnegative integers and by $\mathbb{Q}_+$ the set of nonnegative rational numbers.
We assume familiarity with the basic concepts in graph theory as used, e.g., by West~\cite{MR1367739}.
We denote the vertex set and the edge set of a graph $G$ by $V(G)$ and $E(G)$, respectively.
A graph is \emph{null} if it has no vertices.
The \emph{neighborhood} of a vertex $v$ in $G$, which corresponds to the set of vertices adjacent to $v$ in $G$, is denoted by $N_G(v)$.
The \emph{closed neighborhood} of $v$ is the set $N_G[v] = N_G(v) \cup \{v\}$.
These concepts are extended to sets $X\subseteq V(G)$ so that $N_G[X]$ is defined as the union of all closed neighborhoods of vertices in $X$, and $N_G(X)$ is defined as the set $N_G[X]\setminus X$.
The \emph{degree} of $v$, denoted by $d_G(v)$, is the cardinality of the set $N_G(v)$.
When there is no ambiguity, we may omit the subscript $G$ in the notations of the degree, and open and closed neighborhoods, and thus simply write $d(v)$, $N(v)$, and $N[v]$, respectively.
Given a set $X \subseteq V(G)$, we denote by $G-X$ the graph obtained from $G$ after deleting all the vertices in $X$.
Similarly, given a vertex $v \in V(G)$, we denote by $G-v$ the graph obtained from $G$ after deleting $v$.
Given two disjoint vertex sets $A,B\subseteq V(G)$, we say that $A$ is \emph{anticomplete} to $B$ (and vice versa) if no vertex in $A$ is adjacent to a vertex in $B$.
The fact that two graphs $G$ and $H$ are isomorphic to each other is denoted by $G\cong H$.
The \emph{complement} of a graph $G$ is the graph, denoted by $\overline{G}$, with vertex set $V(G)$, in which two distinct vertices are adjacent if and only if they are non-adjacent in $G$.
For a positive integer $n$, we denote the $n$-vertex complete graph, path, and cycle by $K_n$, $P_n$, and $C_n$, respectively.
Similarly, for $n\ge 2$ we denote by $K_n^{-}$ the graph obtained from the complete graph $K_n$ by deleting an edge, and for positive integers $m$ and $n$ we denote by $K_{m,n}$ the complete bipartite graph with parts of sizes $m$ and $n$.
The graph $K_4^-$ is also called the \emph{diamond}.
For $n\ge 4$, we denote by $W_n$ the \emph{$n$-wheel}, that is, the graph obtained from the graph $C_n$ by adding to it a universal vertex.
A \emph{wheel} is a graph isomorphic to $W_n$ for some $n \geq 4$.

An \emph{independent set} in a graph $G$ is a set of pairwise non-adjacent vertices, and a \emph{clique} is a set of pairwise adjacent vertices.
Given a graph $G$ and a weight function $w:V(G)\to \mathbb{Q}_+$, the \textsc{Max Weight Independent Set} problem asks to find an independent set $I$ in $G$ of maximum possible weight $w(I)$, where $w(I) = \sum_{x\in I}w(x)$.
The \emph{independence number} of $G$, denoted by $\alpha(G)$, is the maximum size of an independent set in~$G$.
The \emph{clique number} of $G$, denoted by $\omega(G)$, is the maximum size of a clique in~$G$.
Given a positive integer $k$, a graph $G$ is \emph{$k$-connected} if it has at least $k+1$ vertices and for every set $X\subseteq V(G)$ with $|X|<k$, the graph $G-X$ is connected.
Given a graph $G$, a set $X\subseteq V(G)$ is a \emph{cutset} in $G$ if $G-X$ is disconnected; a \emph{$k$-cutset} is a cutset of cardinality $k$.
(Note that if $G$ is disconnected, the above definition allows the number of components of $G-X$ to be the same as the number of components of $G$.)
A \emph{cut-vertex} in a connected graph $G$ is a vertex $v$ such that $G-v$ is disconnected.
A \emph{cut-partition} in a graph $G$ is a triple $(A,B,C)$ such that $A\cup B\cup C = V(G)$, $A\neq \emptyset$, $B\neq \emptyset$, and $G$ has no edges with one endpoint in $A$ and one in $B$.
A \emph{clique cutset} in $G$ is a clique $C$ such that $G$ has a cut-partition $(A,B,C)$ for some $A$ and $B$.
Given two non-adjacent vertices $u$ and $v$ in $G$, a \emph{$u{,}v$-separator} is a set $S\subseteq V(G)\setminus \{u,v\}$ such that $u$ and $v$ belong to different components of $G-S$.
A $u{,}v$-separator $S$ is \emph{minimal} if no proper subset of $S$ is a $u{,}v$-separator.
A \emph{minimal separator} in $G$ is a set $S\subseteq V(G)$ that is a minimal $u{,}v$-separator for some non-adjacent vertex pair $u{,}v$.
A well-known characterization of minimal separators (see, e.g.,~\cite{MR2063679}) is that a set $S\subseteq V(G)$ is a minimal separator if and only if the graph $G-S$ contains at least two \emph{$S$-full components}, that is, two components such that every vertex in $S$ has a neighbor in each of them.

We now define the six graph containment relations studied in this paper.
A graph $H$ is said to be an \emph{induced subgraph} of $G$ if $H$ can be obtained from $G$ by deleting vertices. If $H$ is obtained from $G$ by deleting vertices and edges, then $H$ is a \emph{subgraph} of $G$.
The \emph{subdivision of an edge} $uv$ of a graph is the operation that deletes the edge $uv$ and adds a new vertex $w$ and two edges $uw$ and $wv$.
A \emph{subdivision of a graph} $H$ is a graph obtained from $H$ by a sequence of edge subdivisions.
A graph $H$ is said to be a \emph{topological minor} of a graph $G$ if $G$ contains a subdivision of $H$ as a subgraph.
Similarly, $H$ is an \emph{induced topological minor} of $G$ if some subdivision of $H$ is isomorphic to an induced subgraph of $G$.
An \emph{edge contraction} is the operation of deleting a pair of adjacent vertices and replacing them with a new vertex whose neighborhood is the union of the neighborhoods of the two original vertices.
If $H$ can be obtained from $G$ by a sequence of vertex deletions, edge deletions, and edge contractions, then $H$ is said to be a \emph{minor} of $G$. Finally, we say that $G$ contains $H$ as an \emph{induced minor} if $H$ can be obtained from $G$ by a sequence of vertex deletions and edge contractions.
If $G$ does not contain an induced subgraph isomorphic to $H$, then we say that $G$ is \emph{$H$-free}. Analogously, we also say that $G$ is $H$-subgraph-free, $H$-topological-minor-free, $H$-induced-topological-minor-free, $H$-minor-free, or $H$-induced-minor-free, respectively, for the other five relations.
Note that $G$ contains $H$ as an induced minor if and only if there exists an \emph{induced minor model} of $H$ in $G$, that is, a collection $(X_u: u\in V(H))$ of pairwise disjoint subsets of $V(G)$ such that each $X_u$ induces a connected subgraph of $G$ and for every two distinct vertices $u,v\in V(H)$, there is an edge in $G$ between a vertex of $X_u$ and a vertex of $X_v$ if and only if $uv \in E(H)$.

A \emph{tree decomposition} of a graph $G$ is a pair $\mathcal{T} = (T, \{X_t\}_{t\in V(T)})$ such that $T$ is a tree and every node $t$ of $T$ is assigned a \emph{bag}, that is, a vertex subset $X_t\subseteq V(G)$, such that the following conditions are satisfied:
every vertex of $G$ is in at least one bag, for every edge $uv\in E(G)$ there exists a node $t\in V(T)$ such that $X_t$ contains both $u$ and $v$, and for every vertex $u\in V(G)$ the subgraph $T_u$ of $T$ induced by the set $\{t\in V(T): u\in X_t\}$ is connected (that is, a tree).
The \emph{width} of $\mathcal{T}$ is the maximum value of $|X_t|-1$ over all $t\in V(T)$.
The \emph{treewidth} of a graph $G$, denoted by $\tw(G)$, is defined as the minimum width of a tree decomposition of $G$.

A graph $G$ is \emph{chordal} if it has no induced cycles of length at least four.
It is known that a graph $G$ is chordal if and only if it admits a \emph{clique tree}, that is, a tree decomposition of $G$ whose bags are exactly the maximal cliques of $G$ (see, e.g.,~\cite{MR1320296}).
Given a graph $G = (V,E)$, a \emph{triangulation} of $G$ is a chordal graph of the form
$(V,E\cup F)$ where $F$ is a subset of the set of non-edges of $G$.
A triangulation $(V,E \cup F)$ of $G$ is a \emph{minimal triangulation} of $G$ if $G$ has no triangulation $(V,E\cup F')$ such that $F'$ is a proper subset of $F$.

An (integer) \emph{graph invariant} is a mapping from the class of all graphs to the set of nonnegative integers that does not distinguish between isomorphic graphs.
For a graph invariant $\rho$, we say that a graph class $\mathcal{G}$ has \emph{bounded $\rho$} if there exists an integer $k$ such that $\rho(G)\le k$ for every graph $G$ in the class.
A graph class $\mathcal{G}$ is said to be \emph{$(\tw,\omega)$-bounded} if it admits a \emph{$(\tw,\omega)$-binding function}, that is, a function $f$ such that for every graph $G$ in the class and any induced subgraph $G'$ of $G$, the treewidth of $G'$ is at most $f(\omega(G'))$.
Furthermore, $\mathcal{G}$ is said to be \emph{polynomially $(\tw,\omega)$-bounded} if it admits a polynomial $(\tw,\omega)$-binding function.
Ramsey's theorem~\cite{MR1576401} states that for every two positive integers $p$ and $q$ there exists an integer $N(p,q)$ such that every graph with at least $N(p,q)$ vertices contains either a clique of size $p$ or an independent set of size $q$. The least such positive integer is denoted by $R(p,q)$.
Following Raghavan and Spinrad~\cite{MR2006100}, we say that an algorithm for a particular graph problem is \emph{robust} for a graph class $\mathcal{G}$ if for every input graph $G$, the algorithm either solves the problem or determines that the graph is not in $\mathcal{G}$.

\subsection{On the tree-independence number}

Given a tree decomposition $\mathcal{T} = (T, \{X_t\}_{t\in V(T)})$ of a graph $G$, the \emph{independence number} of $\mathcal T$, denoted by $\alpha(\mathcal{T})$, is defined as the maximum independence number of $G[X_t]$ among all nodes $t\in V(T)$.
The \emph{tree-independence number} of $G$, denoted by $\tin(G)$, is the minimum independence number among all possible tree decompositions of $G$.

Known properties of chordal graphs imply that the class of chordal graphs coincides with the class of graphs with tree-independence number at most one (see Theorem 3.3 in~\cite{dallard2022firstpaper}).

\begin{theorem}\label{chordal}
Let $G$ be a graph. Then $\tin(G)\le 1$ if and only if $G$ is chordal.
\end{theorem}

While the tree-independence number can increase upon edge deletions, it is monotone under vertex deletions and edge contractions.

\begin{proposition}[Proposition 3.9 in~\cite{dallard2022firstpaper}]\label{lem:tree-indepencence number induced-minor}
Let $G$ be a graph and $G'$ an induced minor of $G$.
Then $\tin(G')\le \tin(G)$.
\end{proposition}

A possible reason for large tree-independence number is the presence of a large induced balanced complete bipartite subgraph.

\begin{lemma}[Corollary 3.6 in~\cite{dallard2022firstpaper}]\label{tin-of-Knn}
For every positive integer $n$, we have $\tin(K_{n,n}) = n$.
\end{lemma}

Given a graph $G$ and a family $\HH =\{H_j\}_{j\in J}$ of connected subgraphs of $G$, we denote by $G(\HH)$ the graph with vertex set $J$, in which two distinct elements $i,j\in J$ are adjacent if and only if $H_i$ and $H_j$ either have a vertex in common or there is an edge in $G$ connecting them.
A subfamily $\HH'$ of $\HH$ is said to be an \emph{independent $\HH$-packing} in $G$ if every two graphs in $\HH'$ are vertex-disjoint and there is no edge between them, that is, $\HH'$ is an independent set in the graph $G(\HH)$.
Assume now that the subgraphs in $\HH$ are equipped with a weight function $w:J\to \mathbb{Q}_+$ assigning weight $w_j$ to each subgraph $H_j$.
For any set $I\subseteq J$, we define the \emph{weight} of the family $\HH' = \{H_i\}_{i\in I}$ as the sum $\sum_{i\in I}w_i$.
Given a graph $G$, a finite family $\HH = \{H_j\}_{j\in J}$ of connected nonnull subgraphs of $G$, and a weight function $w:J\to \mathbb{Q}_+$ on the subgraphs in $\HH$, the \textsc{Max Weight Independent Packing} problem asks to find an independent $\HH$-packing in $G$ of maximum weight.
Note that the problem is equivalent to the \textsc{Max Weight Independent Set} problem on the derived graph $G(\HH)$.

The construction $G(\HH)$ was considered by Cameron and Hell in~\cite{MR2190818}, who focused on the following particular case.
Given a graph $G$ and a (finite or infinite) set $\FF$ of connected graphs, we denote by $\HH(G,\FF)$ the family of all subgraphs of $G$ isomorphic to a member of $\FF$.
Note that, for $\HH = \HH(G,\FF)$, the graph $G(\HH)$ is isomorphic to $G$ if
$\FF = \{K_1\}$ and to the square of the line graph of $G$ if $\FF = \{K_2\}$.
In the special case of the \textsc{Max Weight Independent Packing} when $\HH = \HH(G,\FF)$, we obtain the \textsc{Max Weight Independent $\mathcal{F}$-Packing} problem, a common generalization of several problems studied in the literature, including the \textsc{Independent $\mathcal{F}$-Packing} problem (see~\cite{MR2190818}), the \textsc{Max Weight Independent Set} problem, the \textsc{Max Weight Induced Matching} problem (see, e.g.,~\cite{MR3776983,MR4151749}), the \textsc{Dissociation Set} problem (see, e.g.,~\cite{MR3593941,MR615221,MR2812599}), and the \textsc{$k$-Separator} problem (see, e.g.,~\cite{MR3987192,MR3296270}).

In the previous work of the series~\cite{dallard2022firstpaper}, we developed a polynomial-time algorithm for the \textsc{Max Weight Independent Packing} problem under the assumption that the input graph is equipped with a tree decomposition of bounded independence number.

\begin{theorem}[Theorem 7.2 in~\cite{dallard2022firstpaper}]\label{thm:bounded-tree-independence-number-packings}
Let $k$ be a positive integer.
Then, given a graph $G$ and a finite family $\mathcal{H} = \{H_j\}_{j \in J}$ of connected nonnull subgraphs of $G$, the \textsc{Max Weight Independent Packing} problem can be solved in time \hbox{$\mathcal{O}(|J| \cdot ((|J| + |V(T)|) \cdot |V(G)| + |E(G)|+|J|^{k}\cdot|V(T)|)$} if $G$ is given together with a tree decomposition \hbox{$\mathcal{T} = (T, \{\Bag_t\}_{t\in V(T)})$} with independence number at most $k$.
\end{theorem}

One of the key steps in the proof of \cref{thm:bounded-tree-independence-number-packings} was a dynamic programming algorithm for the \textsc{Max Weight Independent Set} problem obtained by adapting the standard dynamic programming approach over tree decompositions of bounded width to the more general setting of tree decompositions with bounded independence number.
In fact, an improved version of this algorithmic result was given using the notion of an \emph{$\ell$-refined tree decomposition}, that is, a tree decomposition in which each bag comes equipped with a subset of at most $\ell$ vertices.
Ignoring such vertices when computing the independence number of the subgraph induced by the bag may lead to improved running times (\cref{thm:bounded-ell-refined-tree-independence-number}).
Two particular examples of such improvements are explained in \cref{refined-improvement-1,refined-improvement-2}.

\begin{definition}%
\label{definition-ell-refined}
Given a nonnegative integer $\ell$, an \emph{$\ell$-refined tree decomposition} of a graph $G$ is a pair $\widehat{\mathcal T} = (T, \{(\Bag_t,U_t)_{t \in V(T)}\})$ such that $\mathcal T = (T,  \{\Bag_t\}_{t\in V(T)})$ is a tree decomposition of $G$, and for every $t \in V(T)$ we have $U_t \subseteq X_t$ and $|U_t| \leq \ell$.
We will refer to $\mathcal{T}$ as the \emph{underlying tree decomposition} of~$\widehat{\mathcal T}$.
\end{definition}

Any concept defined for tree decompositions can be naturally extended to $\ell$-refined tree decompositions, simply by considering it on the underlying tree decomposition.

The definition of $\ell$-refined tree decomposition leads to the following refinement of the tree-independence number.

\begin{definition}%
\label{definition-ell-refined-tree-alpha}
Given a nonnegative integer $\ell$, the \emph{residual independence number} of an $\ell$-refined tree decomposition $\widehat{\mathcal T}$ of a graph $G$ is defined as $\max_{t \in V(T)} \alpha(G[X_t \setminus U_t])$ and denoted by $\widehat{\alpha}(\widehat{\mathcal T})$.
The \emph{$\ell$-refined tree-independence number} of a graph $G$ is defined as the minimum residual independence number of an $\ell$-refined tree decomposition of $G$, and denoted by $\ell$\textnormal{-}$\tin(G)$.
\end{definition}

The notion of $\ell$-refined tree decompositions with bounded residual independence number generalizes the notion of $(k,\ell)$-semi clique tree decompositions studied by Jacob, Panolan, Raman, and Sahlot in~\cite{MR4189425}, as well as tree decompositions where each bag is obtained from a clique by deleting a bounded number of edges used by Fomin and Golovach~\cite{MR4276552}.

\begin{observation}[Observation 4.3 in~\cite{dallard2022firstpaper}]\label{observation-ell-tree-alpha}
For every graph $G$ and every integer $\ell\ge 0$, we have
\[\ell\textnormal{-}\tin(G)\le \tin(G)\le \ell\textnormal{-}\tin(G) + \ell\,.\]
In particular, equalities hold when $\ell = 0$, that is, $0\textnormal{-}\tin(G) = \tin(G)$.
\end{observation}

We now state precisely the aforementioned algorithmic results for the \textsc{Max Weight Independent Set} problem.

\begin{theorem}[Theorem 5.2 in~\cite{dallard2022firstpaper}]\label{thm:bounded-ell-refined-tree-independence-number}
For every integer $k\ge 1$, \textsc{Max Weight Independent Set} is solvable in time $\mathcal{O}(2^\ell\cdot |V(G)|^{k+1}\cdot|V(T)|)$ if the input vertex-weighted graph $G$ is given with an $\ell$-refined tree decomposition \hbox{$\widehat{\mathcal{T}} = (T, \{(\Bag_t,U_t)\}_{t\in V(T)})$} with residual independence number at most $k$.
\end{theorem}

The case $\ell = 0$ yields the following.

\begin{corollary}[Corollary 5.3 in~\cite{dallard2022firstpaper}]\label{thm:bounded-tree-independence-number}
For every $k\ge 1$, \textsc{Max Weight Independent Set} is solvable in time $\mathcal{O}(|V(G)|^{k+1}\cdot|V(T)|)$ if the input vertex-weighted graph $G$ is given with a tree decomposition \hbox{$\mathcal{T} = (T, \{\Bag_t\}_{t\in V(T)})$} with independence number at most $k$.
\end{corollary}

Computing $\ell$-refined tree decompositions with bounded residual independence number can be reduced to graphs without clique cutsets; in particular, to $2$-connected graphs.
We will use this result in the proof of~\cref{cor:reduction-to-blocks}.

\begin{proposition}[Proposition 4.6 in~\cite{dallard2022firstpaper}]\label{reduction-to-atoms}
Let $C$ be a clique cutset in a graph $G$ and let $(A,B,C)$ be a cut-partition of $G$.
Let $G_A = G[A\cup C]$ and $G_B = G[B\cup C]$, and let $\widehat{\mathcal{T}}_A$ and $\widehat{\mathcal{T}}_B$ be $\ell$-refined tree decompositions of $G_A$ and $G_B$, respectively.
Then we can compute in time $\mathcal{O}(|\widehat{\mathcal{T}}_A|+|\widehat{\mathcal{T}}_B|)$ an  $\ell$-refined tree decomposition $\widehat{\mathcal{T}}$ of $G$ such that $\widehat{\alpha}(\widehat{\mathcal{T}}) = \max\{\widehat{\alpha}(\widehat{\mathcal{T}}_A),\widehat{\alpha}(\widehat{\mathcal{T}}_B)\}$.
\end{proposition}

The next lemma, showing $(\tw,\omega)$-boundedness, combines Lemmas 3.2 and 4.4 in~\cite{dallard2022firstpaper}.

\begin{lemma}\label{bounded tin implies bounded tw-omega-refined}
For every two nonnegative integers $k$ and $\ell$, the class of graphs with $\ell$-refined tree-independence number at most $k$ is $(\tw,\omega)$-bounded, with a binding function $f(p) = R(p+1,k+1)+\ell-2$, which is upper-bounded by a polynomial in $p$ of degree $k$.\footnote{The fact that for fixed $k$, the Ramsey number $R(p+1,k+1)$ is upper-bounded by a polynomial in $p$ of degree $k$ follows from a well-known upper bound on the Ramsey numbers stating that $R(a,b)\le {a+b-2\choose a-1}$ for all positive integers $a$ and $b$.}
In particular, for every positive integer $k$, the class of graphs with tree-independence number at most $k$ is $(\tw,\omega)$-bounded, with a binding function $f(p) = R(p+1,k+1)-2$.
\end{lemma}

\section{$K_{2,q}$-induced-minor-free graphs}\label{sec:K2q-induced-minor-free}

In this section, we prove boundedness of the tree-independence number in the classes of $K_{2,q}$-induced-minor-free graphs.
We build on the work of Bouchitt\'{e} and Todinca~\cite{MR1896345} and make use of a number of concepts that are often used in the literature in relation with tree decompositions: minimal separators, minimal triangulations, and potential maximal cliques.

\begin{definition}
A \emph{potential maximal clique} in a graph $G$ is a set $X\subseteq V(G)$ such that $X$ is a maximal clique in some minimal triangulation of $G$.
\end{definition}

Our approach can be summarized as follows.
First, we characterize the $K_{2,q}$-induced-minor-free graphs in terms of minimal separators.
Second, using the fact that every graph has an efficiently computable tree decomposition in which every bag is a potential maximal clique, we show that the tree-independence number is bounded and efficiently witnessed in any class of graphs in which there is a bound on the maximum size of an independent set contained in a potential maximal clique.
Finally, we connect this observation with minimal separators by proving that in every graph, every potential maximal clique is either a clique or is contained in the union of two minimal separators (\cref{thm:pmcs-covering}).

\begin{lemma}\label{lem:K2q}
Given a positive integer $q$, a graph $G$ is $K_{2,q}$-induced-minor-free if and only if every minimal separator in $G$ induces a subgraph with independence number less than $q$.
\end{lemma}

\begin{proof}
Fix $q\ge 1$ and a graph $G$.
Let $S$ be a minimal separator in $G$, and let $C$ and $D$ be two $S$-full components of $G-S$.
Note that for any independent set $I$ of $G$ contained in $S$, deleting from $G$ all vertices in $V(G)\setminus (V(C)\cup V(D)\cup I)$ and then contracting all edges fully contained within $C$ or $D$ yields a graph isomorphic to $K_{2,|I|}$, showing that $G$ contains $K_{2,|I|}$ as an induced minor. Thus, if $G$ is $K_{2,q}$-induced-minor-free, then $\alpha(G[S])<q$.

Suppose now that $G$ contains $K_{2,q}$ as an induced minor. Fix a bipartition $\{A,B\}$ of $K_{2,q}$ such that
$A = \{a_1,a_2\}$ and $B = \{b_1,\ldots, b_q\}$, and let $M= (X_u : u\in V(K_{2,q}))$
be an induced minor model of $K_{2,q}$ in $G$ that minimizes the sum $\sum_{b\in B}|X_b|$.
Since each vertex $b\in B$ has degree two in $K_{2,q}$, the minimality of $M$ implies that $|X_b| = 1$ for all $b\in B$.
Let $I = \cup_{b\in B}X_b$ and $W = I\cup X_{a_1}\cup X_{a_2}$.
Since $M$ is an induced minor model of $K_{2,q}$ in $G$, the set $I$ is independent in $G$, while the sets $X_{a_1}$ and $X_{a_2}$ induce connected subgraphs of $G$ and are anticomplete to each other.
In particular, since every vertex in $I$ has a neighbor in $G$ in both $X_{a_1}$ and $X_{a_2}$, we infer that $I$ is a minimal separator in the subgraph of $G$ induced by $W$.
Furthermore, since $I\cup (V(G)\setminus W)$ separates $X_{a_1}$ from $X_{a_2}$ in $G$, there exists a minimal separator $S$ in $G$ such that $I\subseteq S$.
Consequently, $\alpha(G[S])\ge |I| = |B| = q$.
\end{proof}

The importance of potential maximal cliques for our purpose is due to the following upper bound on the tree-independence number.
Given a graph $G$, we denote by $\alpha_{pmc}(G)$ the \emph{pmc-independence number} of $G$, defined as the maximum independence number of a subgraph of $G$ induced by some potential maximal clique in $G$.
Note that the value of $\alpha_{pmc}(G)$ is well-defined, since every graph has at least one potential maximal clique.
In the next lemma, we denote by $\mu$ the \emph{matrix multiplication exponent}, that is, the smallest real number such that two $n\times n$ binary matrices can be multiplied in time $\mathcal{O}(n^{\mu+\epsilon})$ for all $\epsilon>0$. It is known that $\mu < 2.37286$~\cite{MR4262465}.

\begin{lemma}\label{tin-pmcin}
Let $G$ be an $n$-vertex graph. Then $\tin(G)\le \alpha_{pmc}(G)$.
Furthermore, a tree decomposition of $G$ with at most $n$ nodes and with independence number at most
$\alpha_{pmc}(G)$ can be computed in time $\mathcal{O}(n^{\mu}\log n)$.
\end{lemma}

\begin{proof}
First we compute a minimal triangulation $G'$ of $G$; this can be done in time $\mathcal{O}(n^\mu \log{n})$~\cite{MR2206129}.
Since $G'$ is a chordal graph, we use an algorithm due to Berry and Simonet~\cite{MR3632036} to compute in time $\mathcal{O}(|V(G')|+|E(G')|) = \mathcal{O}(n^2)$ a clique tree $\mathcal T = (T, \{X_t : t \in V(T)\})$ of $G'$ with at most $|V(G')| = n$ nodes.
By construction, every bag in this tree decomposition is a maximal clique of $G'$, and hence a potential maximal clique in $G$.
Furthermore, $\mathcal{T}$ is a tree decomposition of $G$.
Since the independence number of this tree decomposition is at most $\alpha_{pmc}(G)$, this is also an upper bound on the tree-independence number of $G$.
\end{proof}

The following theorem due to Bouchitt\'{e} and Todinca~\cite{MR1734049} characterizes potential maximal cliques.

\begin{theorem}\label{thm:pmcs}
Let $G$ be a graph and let $X\subseteq V(G)$. Then $X$ is a potential maximal clique in $G$ if and only if the following two conditions hold:
\begin{enumerate}
  \item For every component $C$ of $G-X$, some vertex of $X$ has no neighbors in $C$.
  \item \label[condition]{non-complete} For every two non-adjacent vertices $u,v\in X$ there exists a component $C$ of $G-X$ in which both $u$ and $v$ have a neighbor.
 \end{enumerate}
Furthermore, if $X$ is a potential maximal clique, then the family of minimal separators $S$ in $G$ such that $S\subset X$ equals
the family of the neighborhoods, in $G$, of the components of $G-X$.
\end{theorem}

Let $X$ be a potential maximal clique in a graph $G$, let $C$ be a component of $G-X$, and let $S = N_G(V(C))$.
We say that $S$ is an \emph{active separator for $X$} if there exist two non-adjacent vertices $u,v\in S$ such that $C$ is the only component of $G-X$ in which $u$ and $v$ both have neighbors. (This definition is formulated slightly differently in~\cite[Definition 13]{MR1896345}.
However, the two definitions are equivalent, due to \cref{thm:pmcs}.) Not every potential maximal clique contains an active separator, but for those that do, the following result holds, as a consequence of~\cite[Theorem 15]{MR1896345}.

\begin{theorem}\label{thm:pmcs-active}
Let $X$ be a potential maximal clique in a graph $G$ and let $S$ be an active minimal separator for $X$. Then there exists a minimal separator $T$ of $G$ such that $X\subseteq S\cup T$.
\end{theorem}

The following is a part of the statement of~\cite[Proposition 19]{MR1896345}.

\begin{proposition}\label{prop:G-a}
Let $a\in V(G)$ and let $X$ be a potential maximal clique of $G$ and such that $a\not\in X$.
Let $C_a$ be the connected component of $G-X$ containing $a$ and let $S= N_G(V(C_a))$.
Then either $X$ is a potential maximal clique of $G-a$ or $S$ is active for $X$.
\end{proposition}

For the next theorem we need the following property of minimal separators.

\begin{lemma}\label{lem:separators}
Let $G$ be a graph, let $G'$ be an induced subgraph of $G$, and let $S'$ be a minimal separator in $G'$. Then there exists a minimal separator $S$ in $G$ such that $S'\subseteq S$.
\end{lemma}

\begin{proof}
Let $C'$ and $D'$ be two $S'$-full components of the graph $G'-S'$.
Fix two vertices $u\in V(C')$ and $v\in V(D')$.
Consider the set $\tilde S = S'\cup (V(G)\setminus V(G'))$.
This set is a $u{,}v$-separator in $G$, and therefore contains a minimal $u{,}v$-separator $S$ in $G$.
Since the set $\tilde S$ is disjoint from $V(C')\cup V(D')$ and $S\subseteq \tilde S$, the set $S$ is also disjoint from $V(C')\cup V(D')$.
Furthermore, since every vertex in $S'$ has a neighbor in both $C'$ and $D'$, we must have $S'\subseteq S$ as otherwise $S$ would not be a $u{,}v$-separator in~$G$.
\end{proof}

\cref{thm:pmcs-active} implies that every potential maximal clique that contains an active minimal separator can be covered by two minimal separators in $G$.
For an arbitrary potential maximal clique, the following holds.

\begin{theorem}\label{thm:pmcs-covering}
Let $G$ be a graph and let $X$ be a potential maximal clique in $G$.
Then $X$ is either a clique in $G$ or there exist two minimal separators $S$ and $T$ in $G$ such that $X\subseteq S\cup T$.
\end{theorem}

\begin{proof}
We use induction on the number of vertices of $G$. If $G$ is chordal, then the only minimal triangulation of $G$ is $G$ itself, and hence the potential maximal cliques of $G$ coincide with its maximal cliques.
So we may assume that $G$ is not chordal.

Consider a potential maximal clique $X$ of $G$.
Assume that $X$ is not a clique.
Then, by \cref{non-complete} of \cref{thm:pmcs}, $X \neq V(G)$.
If $X$ contains an active minimal separator $S$, then by \cref{thm:pmcs-active}, there exists a minimal separator $T$ of $G$ such that $X\subseteq S\cup T$.

Assume now that $X$ does not contain any active minimal separator.
Since $X\neq V(G)$, there exists a vertex $a\in V(G)\setminus X$.
Let $G' = G-a$, let $C_a$ be the connected component of $G-X$ containing $a$, and let $S= N_G(V(C_a))$. Since $S$ is not an active separator for $X$, \cref{prop:G-a} implies that $X$ is a potential maximal clique of $G'$.
By the induction hypothesis, there exist two minimal separators $S'$ and $T'$ in $G'$ such that $X\subseteq S'\cup T'$.
By \cref{lem:separators}, there exist two minimal separators $S$ and $T$ in $G$ such that $S'\subseteq S$ and $T'\subseteq T$.
Thus, $X\subseteq S\cup T$.
\end{proof}

We now return our attention to the class of $K_{2,q}$-induced-minor-free graphs, where $q$ is a positive integer.
The case $q = 1$ is not of particular interest, as every connected $K_{2,1}$-induced-minor-free graph is complete.
However, for all $q\ge 2$, the class of $K_{2,q}$-induced-minor-free graphs contains the class of chordal graphs, and this containment is proper for all $q\ge 3$.

\Cref{lem:K2q,thm:pmcs-covering} imply the following.

\begin{lemma}\label{lem:pmc-upper-bound}
For every integer $q\ge 2$ and every $K_{2,q}$-induced-minor-free graph $G$,
the pmc-independence number of $G$ is at most $2q-2$.
\end{lemma}
The upper bound from \cref{lem:pmc-upper-bound} is not sharp for the case $q = 2$, when we obtain the class of chordal graphs; for a chordal graph $G$, its potential maximal cliques coincide with its maximal cliques, and hence $\alpha_{pmc}(G) \le 1$.
We can also improve the bound for the case $q = 3$.

\begin{lemma}\label{K23-alpha-pmc}
If $G$ is a $K_{2,3}$-induced-minor-free graph, then $\alpha_{pmc}(G)\le 3$.
\end{lemma}

\begin{proof}
Suppose for a contradiction that $G$ is a $K_{2,3}$-induced-minor-free graph that contains an independent set $I$ of size $4$ that is contained in a potential maximal clique $X$.
By \cref{thm:pmcs}, every two non-adjacent vertices in $X$ are in the neighborhood of some component of $G-X$.
We can thus fix, for any two distinct vertices $x,y\in I$, a component $C(x,y)$ of $G-X$ in which both $x$ and $y$ have a neighbor.
Next, we show that these components are pairwise distinct.
Suppose that this is not the case.
Then there exists a component $C$ of $G-X$ with at least three neighbors in $I$.
By \cref{thm:pmcs}, the neighborhood of $V(C)$ is a minimal separator in $G$.
However, by \cref{lem:K2q} no minimal separator in $G$ contains three pairwise non-adjacent vertices, which implies that $|N(V(C))\cap I|\le 2$, a contradiction.
To complete the proof, let us fix, for any two distinct vertices $x,y\in I$, an arbitrary induced $x,y$-path $P(x,y)$ in $G$ such that all internal vertices of $P(x,y)$ belong to $C(x,y)$.
Writing $I = \{x,y,u,v\}$, we now see that $x$, $y$, the path $P(x,y)$, the concatenation of the paths $P(x,u)$ and $P(u,y)$, and the concatenation of the paths $P(x,v)$ and $P(v,y)$ form an induced subgraph of $G$ isomorphic to a subdivision of $K_{2,3}$.
This contradicts the fact that $G$ is $K_{2,3}$-induced-minor-free.
\end{proof}

The bound given by \cref{K23-alpha-pmc} is sharp, as verified by the $6$-cycle.
Indeed, any maximum independent set in the $6$-cycle is a potential maximal clique, showing that $\alpha_{pmc}(C_6)\ge 3$.
We suspect that the bound from \cref{lem:pmc-upper-bound} may in fact not be sharp for any $q\ge 2$, but leave this question for future work.

\Cref{lem:pmc-upper-bound,tin-pmcin} immediately imply the following result.

\begin{theorem}\label{thm:tree-independence-number-K_2q-induced-minor-free-graphs}
For every integer $q\ge 2$ and every $n$-vertex $K_{2,q}$-induced-minor-free graph $G$, the tree-independence number of $G$ is at most $2q-2$.
Furthermore, a tree decomposition of $G$ with at most $n$ nodes and with independence number at most $2q-2$ can be computed in time $\mathcal{O}(n^{\mu}\log n)$, where $\mu < 2.37286$ is the matrix multiplication exponent.
\end{theorem}

We do not know if the bound on the tree-independence number of  $K_{2,q}$-induced-minor-free graphs given by \cref{thm:tree-independence-number-K_2q-induced-minor-free-graphs} is sharp.
In fact, for all we know, it may even be possible that for all $q\ge 2$, the correct bound is $q-1$ instead of $2(q-1)$.
Recall that the tree-independence number of the complete bipartite graph $K_{q-1,q-1}$ is equal to $q-1$ (see \cref{tin-of-Knn}), so the bound of $q-1$, if true, would be sharp.

\begin{remark}\label{K2q-robust}
One can show that the algorithm used to compute a tree decomposition of a $K_{2,q}$-induced-minor-free graph, as given in \Cref{thm:tree-independence-number-K_2q-induced-minor-free-graphs}, can be turned into a robust polynomial-time algorithm, as follows.
First, for a given $n$-vertex graph $G$ we compute a tree decomposition $\mathcal{T}$ of $G$ in time $\mathcal{O}(n^{\mu}\log n)$, as explained in \Cref{tin-pmcin}.
Then, for each bag $X$ of $\mathcal{T}$, we check in polynomial time (where the degree of the polynomial depends on $q$) whether $G[X]$ contains an independent set of size at least $2q-1$.
If that is the case, then the algorithm returns that $G$ is not $K_{2,q}$-induced-minor-free (which follows by \Cref{lem:pmc-upper-bound,tin-pmcin}).
Otherwise, the tree decomposition $\mathcal{T}$ has tree-independence number at most $2q - 2$.\hfill$\blacktriangle$
\end{remark}

\medskip
\Cref{lem:K2q,thm:tree-independence-number-K_2q-induced-minor-free-graphs} give an alternative proof of the fact that any class of graphs in which all minimal separators are of bounded size has bounded tree-independence number (see~\cite[Theorem 3.8]{dallard2022firstpaper}).
We adopt the same notation as in~\cite{dallard2022firstpaper} and denote for a graph $G$ by $\mms(G)$ the maximum size of a minimal separator of $G$ (unless $G$ is complete, in which case we set $\mms(G) = 0$).

\begin{corollary}\label{tin and mms}
For every graph $G$, we have $\tin(G)\le \max\{1,2\cdot\mms(G)\}$.
\end{corollary}

\begin{proof}
The inequality holds for complete graphs, so we may assume that $G$ is not complete and that  for some positive integer $s$, every minimal separator in $G$ has size at most $s$.
Then every minimal separator in $G$ induces a subgraph with independence number less than $s+1$.
By \cref{lem:K2q}, $G$ is $K_{2,s+1}$-induced-minor-free.
Hence, by \cref{thm:tree-independence-number-K_2q-induced-minor-free-graphs} the tree-independence number of $G$ is at most~$2s$.
\end{proof}

\cref{thm:tree-independence-number-K_2q-induced-minor-free-graphs,thm:bounded-tree-independence-number-packings} have the following consequence for the \textsc{Max Weight Independent Packing} problem.

\begin{theorem}\label{MWIHP-K2q}
For every integer $q \geq 2$, given a $K_{2,q}$-induced-minor-free graph $G$ and a finite family $\mathcal{H} = \{H_j\}_{j \in J}$ of connected nonnull subgraphs of $G$, the \textsc{Max Weight Independent Packing} problem can be solved in time \hbox{$\mathcal{O}(|J|\cdot |V(G)|\cdot (|V(G)| + |J|^{2q-2}))$}.
\end{theorem}

\begin{proof}
Let $G$ be an $n$-vertex $K_{2,q}$-induced-minor-free graph, $\HH = \{H_j\}_{j \in J}$ a finite family of connected nonnull subgraphs of $G$, and $w:J\to \mathbb{Q}_+$ a weight function.
By \cref{thm:tree-independence-number-K_2q-induced-minor-free-graphs}, we can compute in time $\mathcal{O}(n^{\mu}\log n)$ a tree decomposition \hbox{$\mathcal{T} = (T, \{\Bag_t\}_{t\in V(T)})$} of $G$ with $\mathcal{O}(n)$ nodes and with independence number at most $2q-2$.
Thus, by \cref{thm:bounded-tree-independence-number-packings}, \textsc{Max Weight Independent Packing} can be solved on $G$ in time \hbox{$\mathcal{O}(|J|\cdot n\cdot (n + |J|^{2q-2}))$}.
\end{proof}

Applying \cref{MWIHP-K2q} to the case when $\HH$ corresponds to the set of all $1$-vertex subgraphs of the input graph $G$, we obtain the following.

\begin{corollary}\label{cor:MWIS-K2q}
For every integer $q\ge 2$, the \textsc{Max Weight Independent Set} problem is solvable in time $\mathcal{O}(n^{2q})$ on $n$-vertex $K_{2,q}$-induced-minor-free graphs.
\end{corollary}

For the case $q = 3$ the bounds on the running time of the algorithm proving \cref{MWIHP-K2q} can be improved by using \cref{K23-alpha-pmc} instead of \cref{lem:pmc-upper-bound}, which, by \cref{tin-pmcin}, improves the upper bound on the tree-independence number of $K_{2,3}$-induced-minor-free graphs from $4$ to $3$.
We thus obtain the following.

\begin{theorem}\label{MWIHP-K23}
Let $G$ be an $n$-vertex $K_{2,3}$-induced-minor-free graph and let $\HH = \{H_j\}_{j \in J}$ be a finite family of connected nonnull subgraphs of $G$.
Then the \textsc{Max Weight Independent Packing} problem can be solved in time $\mathcal{O}(|J|\cdot |V(G)|\cdot (|V(G)| + |J|^{3}))$.
In particular, the \textsc{Max Weight Independent Set} problem is solvable in time $\mathcal{O}(n^{5})$ on $n$-vertex $K_{2,3}$-induced-minor-free graphs.
\end{theorem}

\Cref{MWIHP-K23} answers the question regarding the complexity of the \textsc{Max Independent Set} problem in the class of $1$-perfectly orientable graphs posed by Beisegel, Chudnovsky, Gurvich, Milani\v{c}, and Servatius in~\cite{MR3992956,MR4357319}.
A graph $G$ is said to be \emph{$1$-perfectly orientable} if its edges can be oriented so that no vertex has a pair of non-adjacent out-neighbors.
The class of $1$-perfectly orientable graphs was introduced in 1982 by Skrien~\cite{MR666799} and studied by Bang-Jensen, Huang, and Prisner~\cite{MR1244934} and, more recently, by Hartinger and Milani\v{c}~\cite{MR3647815} and by Bre\v{s}ar, Hartinger, Kos, Milani\v{c}~\cite{MR3853110}.
While the class of $1$-perfectly orientable graphs is known to be a common generalization of the classes of chordal graphs and circular-arc graphs, the structure of graphs in this class remains elusive.
Nevertheless, it is known that every $1$-perfectly orientable graph is $K_{2,3}$-induced-minor-free (see~\cite{MR3647815}), and hence \cref{MWIHP-K23} has the following consequence.

\begin{corollary}
The \textsc{Max Weight Independent Set} problem is solvable in time $\mathcal{O}(n^{5})$ on $n$-vertex $1$-perfectly orientable graphs.
\end{corollary}

This result complements the $\mathcal{O}(n(n\log n + m\log\log n))$ algorithm for the \textsc{Max Weight Clique} problem on $1$-perfectly orientable graphs with $n$ vertices and $m$ edges due to Beisegel, Chudnovsky, Gurvich, Milani\v{c}, and Servatius~\cite{MR3992956,MR4357319} and a linear-time algorithm for the \textsc{List $k$-Coloring} problem in the same class of graphs, which follows from the corresponding algorithm for $(\tw,\omega)$-bounded graph classes due to Chaplick and Zeman~\cite{DBLP:journals/endm/ChaplickZ17} (see also~\cite{MR4332111}) and the fact that the class of $K_{2,3}$-induced-minor-free graphs is $(\tw,\omega)$-bounded (by a computable binding function)~\cite{DMS-WG2020,dallard2021treewidth}.

\section{Tree-independence number and SPQR trees}\label{sec:sufficient}

In this section, we use block-cutpoint trees and SPQR trees to develop a sufficient condition for bounded and efficiently witnessed tree-independence number for graph classes closed under induced topological minors.
For the description and analysis of running times of algorithms for computing tree decompositions, as well as for the numbers of nodes of these decompositions, we use the following definition.
Let $f:\mathbb{Z}_+\times \mathbb{Z}_+\to \mathbb{Z}_+$ be a function.
We say that $f$ is \emph{superadditive} if the inequality
\[
f(x_1+x_2,y_1+y_2)\ge f(x_1,y_1)+f(x_2,y_2)
\]holds for all $x_1,y_1,x_2,y_2\in \mathbb{Z}_+$.
Note that any superadditive function is non-decreasing with respect to both coordinates.
A similar argument shows that $f$ is non-decreasing in the second coordinate.

While the reductions described in this section could all be done with usual tree decompositions, with respect to their independence numbers, we develop the reductions in the more general context of $\ell$-refined tree decompositions, with respect to their residual independence numbers.
This will lead to significant improvements in the running times of the algorithms for the \textsc{Max Weight Independent Set} problem in the classes of $W_4$- and $K_{5}^-$-induced-minor-free graphs
presented in \Cref{W4-induced-minor-free,K5-induced-minor-free} (cf.~\cref{refined-improvement-1,refined-improvement-2}).

\subsection{Reducing the problem to blocks}\label{sec:blocks}

Given a graph $G$, a \emph{block} of $G$ is a maximal connected subgraph of $G$ without a cut-vertex. Every connected graph $G$ has a unique \emph{block-cutpoint} tree, that is, a tree $T$ with $V(T) = \mathcal{B}(G)\cup \mathcal{C}(G)$, where $\mathcal{B}(G)$ is the set of blocks of $G$ and $\mathcal{C}(G)$ is the set of cut-vertices of $G$, such that every edge of $T$ has an endpoint in $\mathcal{B}(G)$ and the other in $\mathcal{C}(G)$,
with $B\in \mathcal{B}(G)$ adjacent to $v\in \mathcal{C}(G)$ if and only if $v\in V(B)$.
Every leaf of the block-cutpoint tree is a block of $G$, called
a \emph{leaf block} of $G$.
As shown by Hopcroft and Tarjan~\cite{hopcroft1973algorithm}, a block-cutpoint tree of a connected graph $G$ can be computed in linear time.

Applying \cref{reduction-to-atoms} to clique cutsets of size at most one, we obtain the following algorithmic result, reducing the problem of computing $\ell$-refined tree decompositions with small residual independence number to the case of $2$-connected graphs in a hereditary graph class $\mathcal{G}$.

\begin{proposition}\label{cor:reduction-to-blocks}
Let $\mathcal{G}$ be a hereditary graph class for which there exist nonnegative integers $k$ and $\ell$
such that for each $2$-connected graph in ${\mathcal G}$, with $n$ vertices and $m$ edges, one can compute in time $f(n,m)$ an $\ell$-refined tree decomposition with at most $g(n,m)$ nodes and residual independence number at most $k$, where $f$ and $g$ are superadditive functions.
Then, for any graph $G$ in $\mathcal{G}$ with $n\ge 1$ vertices and $m$ edges, one can compute in time $\mathcal{O}(n + m + f(2n,m))$ an $\ell$-refined tree decomposition of $G$ with $\mathcal{O}(n+g(2n,m))$ nodes and residual independence number at most $\max\{2-\ell,k\}$.\footnote{In the case $k = 1$, a simple adaptation of the proof shows that the bound $\max\{2-\ell,k\}$ can be replaced with $k$, which matches the bound for all $k\ge 2$.
If one is only interested in preserving the usual tree-independence number (which corresponds to the case $\ell = 0$), then the Hopcroft--Tarjan algorithm~\cite{hopcroft1973algorithm} combined with \cref{reduction-to-atoms} also leads to a bound of $k$ (except if $k = 0$, in which case the bound becomes $1$).}%
\end{proposition}

\begin{proof}
Let $G \in \mathcal{G}$ be a graph with $n\ge 1$ vertices and $m$ edges.
Using Breadth-First Search and the Hopcroft-Tarjan algorithm, we compute in time $\mathcal{O}(n+m)$ the connected components $G_1,\ldots, G_p$ of $G$ and the corresponding block-cutpoint trees $T_1,\dots,T_p$. Let $\mathcal{B}_2(G_i)$ and $\mathcal{B}_3(G_i)$ denote the sets of blocks of $G_i$ with exactly two, resp.\ at least three, vertices.
For each component $G_i$ we compute an $\ell$-refined tree decomposition $\widehat{\mathcal{T}}_{G_i}$ of $G_i$ recursively as follows.

At each step of the recursion, let $G_i'$ be the current graph and $B$ be a leaf block of $G_i'$.
Initially, we take $G_i' = G_i$.
\begin{enumerate}
\item First, if $G_i' = B$, then we consider the following two cases.

If $B \in \mathcal{B}_2(G_i)$, then $B$ is a complete graph of order at most two.
In this case we compute the $\ell$-refined tree decomposition $\widehat{\mathcal{T}}_{G_i'}$ of $G_i'$ consisting of a single node $t$ whose bag $X_t$ is $V(G_i')$ and the set $U_t$ to be any subset of $V(G_i')$ with $\min\{\ell,2\}$ vertices.
We can compute $\widehat{\mathcal{T}}_{G_i'}$ in constant time; note that this $\ell$-refined tree decomposition of $G_i'$ satisfies
\[\widehat{\alpha}(\widehat{\mathcal{T}}_{G_i'}) = \alpha(G_i'[X_t\setminus U_t]) \le |X_t|-|U_t| = 2-\min\{\ell,2\} = \max\{2-\ell,0\}\le \max\{2-\ell,k\}\,.\]
Otherwise, $B \in \mathcal{B}_3(G_i)$ is $2$-connected, and we compute in time
$f(|V(G_i')|, |E(G_i')|)$ an $\ell$-refined tree decomposition $\widehat{\mathcal{T}}_{G_i'}$ of $G_i'$ with at most
$f(|V(G_i')|, |E(G_i')|)$ nodes and residual independence number at most $k\le \max\{2-\ell,k\}$.

\item Second,  if $G_i'\neq B$, let $(X,Y,Z)$ be a cut-partition where $X = V(B)\setminus\{v\}$, $Y = V(G_i')\setminus V(B)$, and $Z = \{v\}$ for the unique cut-vertex $v$ of $G_i'$ contained in $B$.
We then compute an $\ell$-refined tree decomposition $\widehat{\mathcal{T}}_{G_i'}$ of $G_i'$ by recursively computing $\ell$-refined tree decompositions of the block $B$ and the graph $G_i'[Y\cup Z]$, and applying \cref{reduction-to-atoms} to the cut-partition $(X,Y,Z)$.
\end{enumerate}
This recursive procedure takes time
\[
    \mathcal{O}\left(|\mathcal{B}_2(G_i)|\right)+\sum_{B\in \mathcal{B}_3(G_i)}f(|V(B)|, |E(B)|) \, = \, \mathcal{O}\left(|V(T_i)|+\sum_{B\in \mathcal{B}(G_i)}f(|V(B)|, |E(B)|)\right)\,.
\]
The total number of nodes of $\widehat{\mathcal{T}}_{G_i}$ is upper bounded by
\[|\mathcal{B}_2(G_i)|+\sum_{B\in \mathcal{B}_3(G_i)}g(|V(B)|, |E(B)|) \quad \leq \quad  |V(T_i)|+\sum_{B\in \mathcal{B}(G_i)}g(|V(B)|, |E(B)|)\,.\]

Finally, we create an $\ell$-refined tree decomposition $\widehat{\mathcal{T}}$ of $G$ with residual independence number at most \hbox{$\max\{2-\ell,k\}$} by combining the $\ell$-refined tree decompositions $\widehat{\mathcal{T}}_{G_1},\ldots, \widehat{\mathcal{T}}_{G_p}$ of its connected components in the obvious way (for example, by iteratively applying \cref{reduction-to-atoms} with respect to a sequence of cut-partitions involving the empty clique cutset).
The number of nodes of $\widehat{\mathcal{T}}$ is the sum of the numbers of nodes of $\widehat{\mathcal{T}}_{G_i}$ over all $i\in \{1,\ldots, p\}$.
The overall running time of the algorithm is
\[\mathcal{O}\left(\sum_{i = 1}^p|V(T_i)|+\sum_{i = 1}^p\sum_{B\in \mathcal{B}(G_i)}f(|V(B)|, |E(B)|)\right)\,.\]

An inductive argument based on the number of blocks shows that for every connected graph $H$ with at least one cut-vertex (and thus with at least two blocks), we have $|\mathcal{C}(H)| < |\mathcal{B}(H)|\le |V(H)|-1$.
Thus, for all $i\in \{1,\ldots, p\}$, we have
$|V(T_i)| = |\mathcal{B}(G_i)|+|\mathcal{C}(G_i)|\le 2|V(G_i)|-3$, which implies
$\sum_{i = 1}^p|V(T_i)| = \mathcal{O}(n)$.
Furthermore, for all $i\in \{1,\ldots, p\}$, we have
\begin{align*}
  \sum_{B\in \mathcal{B}(G_i)}|V(B)|
  &= |V(G_i)|+\sum_{v\in \mathcal{C}(G_i)}(d_{T_i}(v)-1)\\
&= |V(G_i)|+|E(T_i)|-|\mathcal{C}(G_i)|\\
&= |V(G_i)|+|V(T_i)|-1-|\mathcal{C}(G_i)|\\
&= |V(G_i)|+|\mathcal{B}(G_i)|-1\\
&\le 2|V(G_i)|-2
\end{align*}
and
\[\sum_{B\in \mathcal{B}(G_i)}|E(B)|\le \sum_{B\in \mathcal{B}(G_i)}|E(B)| = |E(G_i)|\,.\]
Since $f$ is superadditive and non-decreasing in the first coordinate, we have
\begin{align*}
\sum_{B\in \mathcal{B}(G_i)}f(|V(B)|, |E(B)|)
&\le f\left(\sum_{B\in \mathcal{B}(G_i)}|V(B)|, \sum_{B\in \mathcal{B}(G_i)}|E(B)|\right)\\
&\le f\left(2|V(G_i)|-2, |E(G_i)|\right)\,.
\end{align*}
Applying the superadditivity and monotonicity properties of $f$ once again, we obtain
\begin{align*}
\sum_{i = 1}^p f\left(2|V(G_i)|-2, |E(G_i)|\right)
&\le f\left(2\sum_{i = 1}^p|V(G_i)|-2p, \sum_{i = 1}^p|E(G_i)|\right)\\
&\le f\left(2n, m\right),
\end{align*}
and the running time of the algorithm is $\mathcal{O}(n+m+f(2n,m))$, as claimed.

The number of nodes of $\widehat{\mathcal{T}}$ is at most
\[\sum_{i = 1}^p\left(|V(T_i)|+\sum_{B\in \mathcal{B}(G_i)}g(|V(B)|, |E(B)|)\right).\]
Applying the same arguments as we did above for the function $f$ to the function $g$ leads to an upper bound for the number of nodes of $\widehat{\mathcal{T}}$ given by $\mathcal{O}(n+g(2n,m))$.
\end{proof}

\Cref{cor:reduction-to-blocks} leads to a sharp upper bound on the tree-independence number of block-cactus graphs, which is needed in the proof of \cref{dichotomy tin}.
A \emph{block-cactus graph} is a graph every block of which is a cycle or a complete graph.
To derive the upper bound, we first show the following lemma.

\begin{lemma}\label{tin-cycles}
For every integer $n\ge 4$, the tree-independence number of the cycle $C_n$ is exactly $2$.
\end{lemma}

\begin{proof}
Since $C_n$ is not a chordal graph, \cref{chordal} implies that $\tin(C_n)\ge 2$.
Let $v_1,\ldots, v_n$ be an order of the vertices of $C_n$ along the cycle.
We construct a tree decomposition $\mathcal T = (T, \{X_t : t \in V(T)\})$ of $C_n$.
The tree $T$ is an $(n-2)$-vertex path $(t_1,\ldots, t_{n-2})$; for each $i\in \{1,\ldots, n-2\}$, bag $X_{t_i}$ consists of vertices $\{v_i,v_{i+1},v_n\}$.
Note that in every bag $X$ there exist two consecutive vertices of $C_n$ in the order, and hence $G[X]$ contains at least one edge.
Furthermore, each bag $X$ contains exactly $3$ vertices.
This readily implies that $\alpha(G[X]) \leq |X| - 1 = 2$.
Hence, $\alpha(\mathcal T) \leq 2$ and consequently $\tin(C_n)\le 2$.
\end{proof}

Following the fact that every complete graph has tree-independence number at most one,
\cref{cor:reduction-to-blocks,tin-cycles} imply the following.

\begin{corollary}\label{cor:block-cactus-graphs}
The tree-independence number of any block-cactus graph is at most $2$.
\end{corollary}

\subsection{Preliminaries on SPQR trees}

We use a classical tool for decomposing $2$-connected graphs into triconnected components, the SPQR trees.
The SPQR tree data structure was introduced by Di Battista and Tamassia in 1990~\cite{10.1007/BFb0032061} (see also~\cite{MR1377386,DBLP:conf/gd/GutwengerM00,MR327391,DBLP:conf/latin/ReedL08}); however, the underlying decomposition of a $2$-connected graph into its triconnected components has been used already by Mac Lane in 1937~\cite{MR1546002} and developed in detail by Tutte~\cite{MR0210617,MR746795} (see also~\cite{MR586989,MR3577069,MR3385727,MR3213840} for generalizations).
Most applications of SPQR trees in the literature are related to questions involving planarity~\cite{MR1408894,MR1546002,MR973566}, the Steiner tree problem~\cite{MR973566}, and graph drawing~\cite{MR2080688}.
More recent applications also involve improved recognition algorithms of certain well-structured graph classes~\cite{MR3804773} and layered pathwidth in minor-closed graph classes~\cite{MR4129002}.

We use the notation of Dujmovi\'{c}, Eppstein, Joret, Morin, and Wood~\cite{MR4129002}.
Given a $2$-connected graph $G$, an \emph{SPQR tree} of $G$ is a tree $S$ with three types of nodes: \textnormal{P}-nodes, \textnormal{R}-nodes, and \textnormal{S}-nodes.
Each node $a$ is associated with a set $X_a\subseteq V(G)$.
If $a$ is a \textnormal{P}-node, then $X_a$ is a $2$-cutset in $G$.
If $a$ is an \textnormal{R}-node or an \textnormal{S}-node, then $X_a$ is the vertex set of a graph $G_a$ associated to node $a$.
For every \textnormal{S}-node $a$, the graph $G_a$ is a cycle, while for every \textnormal{R}-node $a$, the corresponding graph $G_a$ is $3$-connected.
Every edge in $S$ has exactly one endpoint that is a \textnormal{P}-node.
For every two adjacent nodes $a$ and $b$ of $S$ such that $b$ is a \textnormal{P}-node, the two vertices in $X_b$ are adjacent in $G_a$.
An edge $uv$ in $G_a$ such that $uv \notin E(G)$ is called a \emph{virtual edge}; otherwise, it is called a \emph{real edge}.
For every virtual edge $uv$ in $G_a$ there exists a \textnormal{P}-node $b$ adjacent to $a$ in $S$ such that $X_b = \{u,v\}$.
No two \textnormal{P}-nodes are associated with the same $2$-cutset of $G$.
We denote the sets of all \textnormal{P}-nodes, \textnormal{R}-nodes, and \textnormal{S}-nodes of $S$ by $N_{\text{P}}(S)$, $N_{\text{R}}(S)$, and $N_{\text{S}}(S)$, or simply by $N_{\text{P}}$, $N_{\text{R}}$, and $N_{\text{S}}$, respectively, if the SPQR tree $S$ is clear from the context.
An SPQR tree need not be unique.
For example, if $G$ is the cycle $C_n$, then an SPQR tree of $G$ can consist of a single node $a$, which is an \textnormal{S}-node such that $G_a = G$.
At the other extreme, an SPQR tree of $G$ could be a tree with as many as $n-2$ \textnormal{S}-nodes, each with the corresponding graph $G_a$ isomorphic to the cycle $C_3$, and these \textnormal{S}-nodes are connected to each other via $n-3$ \textnormal{P}-nodes each corresponding to a virtual edge.
There are also many other intermediate possibilities.
However, if we consider an SPQR tree of $G$ with the minimum total number of nodes, then such a tree $S$ can be uniquely obtained, by a simple transformation, from a rooted SPQR tree as originally defined by Di~Battista and Tamassia~\cite{10.1007/BFb0032061}.
In this case, the graphs $G_a$ where $a$ is an \textnormal{R}- or \textnormal{S}-node are called \emph{triconnected components} of $G$ and are uniquely defined (see~\cite{MR1377386,MR1546002,MR327391}).
From now on, we assume this uniqueness property and refer to \emph{the} SPQR tree of $G$.

Thus, given a $2$-connected graph $G$, the SPQR tree $S$ of $G$ represents the collection of triconnected components of $G$ that are joined at $2$-cutsets (\textnormal{P}-nodes).
This structure can be described recursively as follows.
The SPQR tree $S$ is a single-vertex tree with a node $a$ if and only if $G$ is either $3$-connected, in which case $a$ is an \textnormal{R}-node, or $G$ is a cycle, in which case $a$ is an \textnormal{S}-node; in both cases, $G_a = G$.
Otherwise, $S$ has a leaf node $a$ that is either an \textnormal{R}-node or an \textnormal{S}-node such that its unique neighbor $b$ in $S$ is a \textnormal{P}-node.
The \textnormal{P}-node $b$ corresponds to a $2$-cutset $X_b$ of $G$ separating $X_a\setminus X_b$ from $V(G)\setminus X_a$.
Furthermore, the tree $S'$ obtained from $S$ by deleting $a$, as well as its neighbor $b$ if it becomes a leaf, is the SPQR tree of the $2$-connected graph $G'$ obtained from $G$ by deleting the vertices in $X_a\setminus X_b$ and adding the edge between the two vertices in $X_b$ if it does not yet exist.
In particular, since there exists a path in $G$ connecting the two vertices of $X_b$ and having all internal vertices in $V(G_a)\setminus X_b$, there exists a subdivision of $G'$ that is an induced subgraph of $G$ (hence, $G'$ is an induced topological minor of $G$).
For each \textnormal{R}-node or \textnormal{S}-node $a'$ of $S'$, the graph $G_{a'}$ associated with node $a'$ is the same as the corresponding graph associated with $a'$ in $S$.
An inductive argument on the number of nodes of the SPQR tree shows the following.

\begin{lemma}\label{G_a-induced-topological-minors}
Let $G$ be a $2$-connected graph and $S$ the SPQR tree of $G$.
For each node $a$ of $S$ that is either an \textnormal{R}-node or an \textnormal{S}-node, the graph $G_a$ is an induced topological minor of $G$.
\end{lemma}

The SPQR tree of any graph $G$ has $\mathcal{O}(|V(G)|)$ nodes (see~\cite{MR1377386}).
As shown by Gutwenger and Mutzel~\cite{DBLP:conf/gd/GutwengerM00}, the SPQR tree of $G$ can be computed in linear time.
The algorithm by Gutwenger and Mutzel computes a slightly different and rooted variant of SPQR tree, but that tree can easily be turned into the SPQR tree with $\mathcal{O}(|V(G)|)$ nodes and satisfying the properties from our definition in linear time.
Note that the SPQR tree $S$ of a graph $G$ naturally defines a tree decomposition $\mathcal{T}$ of $G$ simply by assigning to each node $a\in V(S)$ the set $X_a$ as the corresponding bag.

For a graph $G$ we denote by $\eta(G)$ the \emph{Hadwiger number} of $G$, that is, the maximum integer $p$ such that $G$ contains $K_p$ as a minor.
Since the Hadwiger number is monotone under taking minors, we readily obtain the following inequality (which we will use in the following sections).

\begin{corollary}\label{corollary-eta}
    Let $G$ be a graph, $S$ the SPQR tree of $G$, and $a$ an \textnormal{R}-node or \textnormal{S}-node of $S$.
    Then
    $\eta(G) \geq \eta(G_a)$.
\end{corollary}

\begin{lemma}\label{SPQR-sum-of-bags}
Let $G$ be a $2$-connected $n$-vertex graph and $S$ the SPQR tree of $G$.
Then
\[
\sum_{a\in N_{\text{R}}\cup N_{\text{S}}}|X_a| \le 3n-6
 \hspace{1cm}\text{~and~}\hspace{1cm}\sum_{a\in N_{\text{R}}}|X_a| \le 2n-4\,.
\]
\end{lemma}

\begin{proof}
We use induction on the number of nodes of $S$.
Assume first that $S$ is a single-vertex tree.
Then $G = G_\ell$ for the unique \textnormal{R}-node or \textnormal{S}-node $\ell$ of $S$, and $|X_\ell| = n$.
If $\ell$ is an \textnormal{R}-node, then $n\ge 4$ and thus $n\le 2n-4$.
If $\ell$ is an \textnormal{S}-node, then $n\ge 3$ and thus $n\le 3n-6$.

Consider now the general case.
Then $S$ has a leaf node $\ell$ that is either an \textnormal{R}-node or an \textnormal{S}-node such that its unique neighbor $p$ in $S$ is a \textnormal{P}-node.
Let $G'$ be the graph obtained from $G$ by deleting the vertices in $X_\ell\setminus X_p$ and adding the edge between the two vertices in $X_p$ if it does not yet exist, and let $S'$ be the tree obtained from $S$ by deleting $\ell$ as well as its neighbor $p$ if it becomes a leaf.
Then $S'$ is the SPQR tree of $G'$.
Let $n'= |V(G')|$.
By the induction hypothesis, we have
\[\sum_{a\in N_{\text{R}}(S')\cup N_{\text{S}}(S')}|X_a| \le 3n'-6\hspace{1cm}\text{~and~}\hspace{1cm}\sum_{a\in N_{\text{R}}(S')}|X_a| \le 2n'-4\,.\]
Note that
\[\sum_{a\in N_{\text{R}}(S)\cup N_{\text{S}}(S)}|X_a| = \sum_{a\in N_{\text{R}}(S')\cup N_{\text{S}}(S')}|X_a| + |X_\ell|\]
and
\[\sum_{a\in N_{\text{R}}(S)}|X_a| \le \sum_{a\in N_{\text{R}}(S')}|X_a| + |X_\ell|\,.\]
Furthermore, we have
\[n' = n - |X_\ell\setminus X_p|
= n - |X_\ell|+2\,.\]
Since $|X_\ell|\ge 3$, we thus obtain
\begin{align*}
\sum_{a\in N_{\text{R}}(S)\cup N_{\text{S}}(S)}|X_a| &=  \sum_{a\in N_{\text{R}}(S')\cup N_{\text{S}}(S')}|X_a| + |X_\ell|\\
&\le 3n' - 6+|X_\ell|\\
& = 3(n - |X_\ell|+2)-6 + |X_\ell|\\
& = 3n - 2|X_\ell|\\
&\le 3n - 6\,.
\end{align*}
Furthermore, if $\ell$ is an \textnormal{S}-node, then
\begin{align*}
\sum_{a\in N_{\text{R}}(S)}|X_a| = \sum_{a\in N_{\text{R}}(S')}|X_a|
\le 2n' - 4
\le 2n - 4\,,
\end{align*}
and if $\ell$ is an \textnormal{R}-node, then $|X_\ell|\ge 4$ and thus
\begin{align*}
\sum_{a\in N_{\text{R}}(S)}|X_a| &\le  \sum_{a\in N_{\text{R}}(S')}|X_a| + |X_\ell|\\
&\le 2n' - 4+|X_\ell|\\
& = 2(n - |X_\ell| + 2) - 4 + |X_\ell|\\
& = 2n - |X_\ell|\\
&\le 2n - 4\,.\qedhere
\end{align*}
\end{proof}

Additionally, we use the following bound on the total number of edges in the graphs $G_a$.

\begin{lemma}[Hopcroft and Tarjan~\cite{MR327391}]\label{SPQR-sum-of-edges}
Let $G$ be a $2$-connected graph and $S$ the SPQR tree of $G$.
Then \[\sum_{a\in N_{\text{R}}\cup N_{\text{S}}}|E(G_a)| \le 3|E(G)|-6\,.\]
\end{lemma}

\subsection{Reducing the problem to triconnected components}\label{triconnected}

Let us first observe that even if every triconnected component of a graph has bounded independence number, the graph does not necessarily have bounded tree-independence number.

\begin{example}\label{example subdivided Kn}
    Let $n \geq 4$ and let $G$ be the graph obtained from a complete graph $K_n$ by subdividing every edge once.
    Observe that the SPQR tree $S$ of $G$ contains a unique \textnormal{R}-node $a$ with $G_a \cong K_n$ to which every \textnormal{P}-node is adjacent, where the \textnormal{P}-nodes correspond to the $2$-cutsets that separate the vertices with degree $2$ from the rest of the graph (see \cref{fig:my_label}).
    In particular, every \textnormal{P}-node $a$ is adjacent to exactly one \textnormal{S}-node $a'$, and $a'$ is a leaf of the SPQR tree $S$ such that the corresponding graph $G_{a'}$ is a cycle on $3$ vertices.
    Hence, every triconnected component of $G$ has independence number~$1$.
    However, $\tin(G) \ge \floor{n/2}$.
    Indeed, since $G$ contains every graph with at most $n$ vertices as an induced topological minor, this in particular implies that the complete bipartite graph $K_{\floor{n/2},\floor{n/2}}$ is an induced topological minor of $G$.
    By \cref{tin-of-Knn,lem:tree-indepencence number induced-minor}, we obtain that $\tin(G) \ge \tin(K_{\floor{n/2},\floor{n/2}}) = \floor{n/2}$.
\end{example}

\begin{figure}[ht]
    \centering
\begin{tikzpicture}[]
\tikzset{every node/.style={draw,circle,fill=black,inner sep=0pt,minimum size=4.5pt}}

\begin{scope}[xscale=0.75]
    \node[label={below:$\mathstrut a$}] (a) at (-1.5,0) {};
    \node[label={below:$\mathstrut b$}] (b) at (-0.5,0) {};
    \node[label={below:$\mathstrut c$}] (c) at (0.5,0) {};
    \node[label={below:$\mathstrut d$}] (d) at (1.5,0) {};

    \node (ab) at (-2.5,2) {};
    \node (ac) at (-1.5,2) {};
    \node (ad) at (-0.5,2) {};

    \draw (a) to (ab) to (b)
          (a) to (ac) to (c)
          (a) to (ad) to (d);

    \node (bc) at (0.5,2) {};
    \node (bd) at (1.5,2) {};

    \draw (b) to (bc) to (c)
          (b) to (bd) to (d);

    \node (cd) at (2.5,2) {};

    \draw (c) to (cd) to (d);
    \end{scope}

\tikzset{
  pics/Pnode/.style n args={3}{
    code = {
    \node[fill=none,ellipse,inner sep=2.5pt] (P#2#3) at (#1,1.75) {
    \begin{tikzpicture}[anchor=center]
    \node[draw=none,fill=none,rectangle] at (-0.2,0) {$\mathstrut #2$};
    \node[draw=none,fill=none,rectangle] at (0.2,0) {$\mathstrut #3$};
    \end{tikzpicture}
    };
    \draw (P#2#3) to ($(R)+(#1/3,0)$);
  }
}
}

\tikzset{
  pics/Snode/.style n args={3}{
    code = {
    \node[fill=none,ellipse] (S#2#3) at (#1,3.35) {
    \begin{tikzpicture}[anchor=center]
    \node[label={below:$\mathstrut #2$}] (P#2#3#2) at (-0.25,0) {};
    \node[label={below:$\mathstrut #3$}] (P#2#3#3) at (0.25,0) {};
    \node (P#2#3x) at (0,0.25) {};
    \draw (P#2#3#2) to (P#2#3#3) to (P#2#3x) to (P#2#3#2);
    \end{tikzpicture}
    };
    \draw (P#2#3) to (S#2#3);
  }
}
}

\begin{scope}[xshift=7.5cm]
    \coordinate (R) at (0,0);

    \draw pic {Pnode={-3.75}{a}{b}};
    \draw pic {Pnode={-2.25}{a}{c}};
    \draw pic {Pnode={-0.75}{a}{d}};
    \draw pic {Pnode={0.75}{b}{c}};
    \draw pic {Pnode={2.25}{b}{d}};
    \draw pic {Pnode={3.75}{c}{d}};

    \draw pic {Snode={-3.75}{a}{b}};
    \draw pic {Snode={-2.25}{a}{c}};
    \draw pic {Snode={-0.75}{a}{d}};
    \draw pic {Snode={0.75}{b}{c}};
    \draw pic {Snode={2.25}{b}{d}};
    \draw pic {Snode={3.75}{c}{d}};

    \draw [decorate,decoration={brace,amplitude=5pt},xshift=5pt,yshift=0pt] (4.55,3.90) -- (4.55,2.80) node [draw=none,fill=none,rectangle,midway,xshift=7pt,label={[align=left]right:{S-nodes}}] {};

    \draw [decorate,decoration={brace,amplitude=5pt},xshift=5pt,yshift=0pt] (4.55,2.1) -- (4.55,1.4) node [draw=none,fill=none,rectangle,midway,xshift=7pt,label={[align=left]right:{P-nodes}}] {};

    \node[rectangle,fill=none,draw=none] at (3.25,0) {R-node};

    \node[fill=white,ellipse] (R) at (0,0) {
    \begin{tikzpicture}[anchor=center]
    \node[label={below:$\mathstrut a$}] (a) at (-1.5,0) {};
    \node[label={below:$\mathstrut b$}] (b) at (-0.5,0) {};
    \node[label={below:$\mathstrut c$}] (c) at (0.5,0) {};
    \node[label={below:$\mathstrut d$}] (d) at (1.5,0) {};
    \draw (a) to (b)
          (a) to[bend left=30] (c)
          (a) to[bend left=45] (d)
          (b) to (c)
          (b) to[bend left=30] (d)
          (c) to (d);
    \end{tikzpicture}};

\end{scope}
\end{tikzpicture}
    \caption{On the left, the graph $G$ obtained from subdividing every edge of $K_4$ once, as in \cref{example subdivided Kn}. On the right, a schematic representation of the SPQR tree of $G$.}
    \label{fig:my_label}
\end{figure}
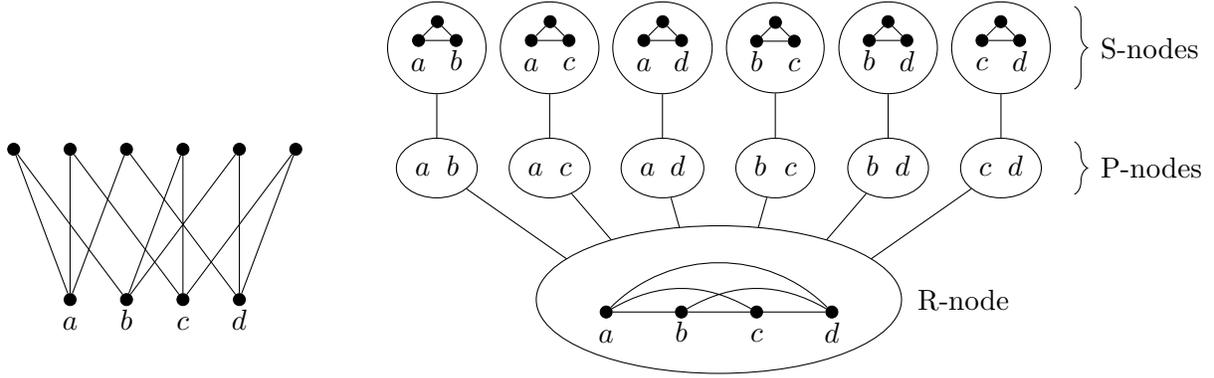

As indicated by \cref{example subdivided Kn}, the reason for the increase in the tree-independence number of a graph when compared to the tree-independence numbers of its triconnected components is due to the fact that subdividing edges can increase the tree-independence number.
Informally speaking, this is due to the possibly complicated structure of the virtual edges present in the triconnected components of the graph.

Let us explain this in more detail.
Given a $2$-connected graph $G$ and the SPQR tree $S$ of $G$, as observed in \cref{G_a-induced-topological-minors}, for every \textnormal{R}-node $a$ of $S$, the triconnected component $G_a$ of $G$ is an induced topological minor of $G$.
Furthermore, the subgraph of $G$ induced by $V(G_a)$ is a spanning subgraph of $G_a$ and every edge in $G_a$ that is not an edge of $G$ is a virtual edge.
Assuming that $G_a$ has bounded tree-independence number, we would like to infer a similar conclusion for the subgraph $G'$ of $G$ induced by $V(G_a)$.
We can do this whenever the set $F$ of virtual edges within $G_a$ can be covered with a bounded number $\ell$ of vertices (see \cref{def:F-cover}).
Indeed, deleting any such set of vertices from $G_a$ results in an induced subgraph of $G$, which means that this part of $G$ will have bounded tree-independence number.
Since only a bounded number of vertices is deleted, the tree-independence number of the subgraph of $G$ induced by $V(G_a)$ will remain bounded.

\begin{definition}\label{def:F-cover}
Let $G$ be a graph, $F^*\subseteq E(G)$ and $G'$ a subgraph of $G$.
We say that a set $U\subseteq V(G')$ is an \emph{$F^*\!$-cover} of $G'$ if $U$ contains at least one endpoint of every edge in $E(G') \cap F^*$.
\end{definition}

The meaning and usefulness of the following definitions will become clearer later, in the proof of \cref{lem:sufficient}.

\begin{sloppypar}
\begin{definition}\label{def:F-compatible ell-refined tree decomposition}
Let $G$ be graph and $F^* \subseteq E(G)$.
An $\ell$-refined tree decomposition $\widehat{\mathcal{T}} = (T,\{(X_t,U_t)_{t \in V(T)}\})$ of $G$ is said to be \emph{$F^*\!$-covering} if
for every node $t \in V(T)$ the set $U_t$ is an $F^*\!$-cover of $G[X_t]$.
\end{definition}
\end{sloppypar}

\begin{definition}\label{def:F-safe}
Let $\mathcal G$ be a graph class, let $G$ be a graph in $\mathcal G$, and let $F^* \subseteq E(G)$.
We say that $F^*$ is a \emph{$\mathcal G$-safe} set of edges of $G$ if deleting from $G$ any subset of $F^*$ results in a graph that belongs to $\mathcal G$.
\end{definition}

\begin{definition}
Let $G$ be a graph, $F \subseteq E(G)$, and $\mathcal T$ a tree decomposition (either usual or $\ell$-refined) of $G$.
An \emph{$F$-mapping} of $\mathcal T$ is a mapping with domain $F$ assigning to every edge $e\in F$ a node $t_e\in V(T)$ such that $e\subseteq X_{t_e}$.
\end{definition}

Note that for any graph $G$, any set $F \subseteq E(G)$, and any (usual or $\ell$-refined) tree decomposition $\mathcal T$ of $G$, there exists an $F$-mapping of $\mathcal T$, since the endpoints of every edge in $G$ are contained in some bag of~$\mathcal T$.
Furthermore, it is clear that an $F$-mapping of $\mathcal T$ can be computed in polynomial time.
However, for a more efficient computation of an $F$-mapping of $\mathcal T$, it may be best to compute it together with $\mathcal T$.
To this end, we consider the following property of a graph class and an associated problem.

\begin{definition}
Given two nonnegative integers $k$ and $\ell$, a graph class $\mathcal{G}$ is said to be \emph{$(k,\ell)$-tree decomposable} if for every $3$-connected graph $G \in \mathcal{G}$ and any $\mathcal{G}$-safe set $F^*$ of edges of $G$, there exists an \hbox{$F^*\!$-covering} $\ell$-refined tree decomposition of $G$ with residual independence number at most $k$.
\end{definition}

\problemdef{$(k,\ell)$-Tree Decomposition($ \mathcal{G})$}{A $3$-connected graph $G \in \mathcal{G}$ and two sets of edges $F^*\subseteq F\subseteq E(G)$ such that $F^*$ is $\mathcal{G}$-safe.}{An $F^*\!$-covering $\ell$-refined tree decomposition $\widehat{\mathcal{T}}$ of $G$ with residual independence number at most $k$ and an $F$-mapping of $\widehat{\mathcal{T}}$.}

For a $(k,\ell)$-tree decomposable graph class $\mathcal{G}$ closed under induced topological minors, the following key lemma reduces the problem of computing $\widehat{\alpha}$-bounded $\ell$-refined tree decompositions of $2$-connected graphs in $\mathcal{G}$ to the {\sc $(k,\ell)$-tree Decomposition($\mathcal{G})$} problem.

\begin{sloppypar}
\begin{lemma}\label{lem:sufficient}
Let $\mathcal{G}$ be a graph class closed under induced topological minors for which there exist nonnegative integers $k$ and $\ell$ such that $\mathcal{G}$ is $(k,\ell)$-tree decomposable and {\sc $(k,\ell)$-tree Decomposition($\mathcal{G})$} can be solved in time $f(n,m)$ on graphs with $n$ vertices and $m$ edges so that the resulting $\ell$-refined tree decomposition has $g(n,m)$ nodes, where $f$ and $g$ are superadditive functions.
Then, for any $2$-connected graph $G$ in $\mathcal{G}$, with $n$ vertices and $m$ edges, one can compute in time $\mathcal{O}(m + f(2n,3m))$ an $\ell$-refined tree decomposition of $G$ with $\mathcal{O}(n + g(2n,3m))$ nodes and residual independence number at most $\max\{3 - \ell,k\}$.
\end{lemma}
\end{sloppypar}

\begin{proof}
Let $G$ be a $2$-connected graph in $\mathcal{G}$ with $n$ vertices and $m$ edges.
We compute in time $\O(n + m)$ the SPQR tree $S$ of $G$.
Let $\widehat{\mathcal{T}} = (T, \{(X_t,U_t) : t \in V(T)\})$ be the $0$-refined tree decomposition of $G$ corresponding to the SPQR tree $S$.
Note that $\widehat{\mathcal{T}}$ is $\ell$-refined.
Our goal is to obtain an $\widehat{\alpha}$-bounded $\ell$-refined tree decomposition of $G$ by updating $\widehat{\mathcal{T}}$ iteratively, as follows.

First, we iterate over all \textnormal{P}-nodes $b$ in $S$ and set $U_b$ to be any subset of $X_b$ with cardinality $\min\{\ell,2\}$.
Since the number of \textnormal{P}-nodes is $\mathcal{O}(n)$, this modification takes time $\mathcal{O}(n)$.
Furthermore, for each \textnormal{P}-node $b$ we then have
\[\alpha(G[X_b\setminus U_b]) \le |X_b|-|U_b| = 2-\min\{\ell,2\} = \max\{2-\ell,0\}\le \max\{3-\ell,k\}\,.\]

Then, for each \textnormal{R}-node or \textnormal{S}-node $a$ in $S$ we compute an $\widehat{\alpha}$-bounded $\ell$-refined tree decomposition $\widehat{\mathcal{T}}_a$ of $G_a$ (using the assumption of the lemma in the case of \textnormal{R}-nodes) and replace the corresponding node in $\widehat{\mathcal{T}}$ with $\widehat{\mathcal{T}}_a$.
Let us describe the update procedure in detail.
For each node $a$ of $S$ that is an \textnormal{R}-node or an \textnormal{S}-node, we perform the following three steps.

\medskip
\noindent{\bf Step 1.} We compute two sets of edges $F_a^*$ and $F_a$ such that $F_a^*\subseteq F_a\subseteq E(G_a)$.
Recall that node $a$ is adjacent in $S$ to \textnormal{P}-nodes only, and for every neighbor $b$ of $a$ in $S$, the set $X_b$ corresponding to the node $b$ is a $2$-cutset in $G$ such that the two vertices in $X_b$ are adjacent in $G_a$.
We define $F_a = \{X_b: b \text{~is adjacent to~} a \text{~in~}S\}$ and $F_a^* = \{e \in F_a : e \text{~is a virtual edge in~} G_a\}$.
Clearly, $F_a$ and $F^*_a$ can be obtained in time $\O(d_S(a))$.

\medskip
The second step relies on the following property of the sets $F^*_a$.

\begin{claim}\label{claim:virtually-safe}
For each \textnormal{R}-node or \textnormal{S}-node $a$ of $S$, the set $F^*_a$ is a $\mathcal{G}$-safe set of edges of $G_a$.
\end{claim}

\begin{subproof}
By \cref{G_a-induced-topological-minors}, there exists a subdivision $G_a'$ of $G_a$ that is an induced subgraph of $G$.
We want to show that deleting from $G_a$ any subset of edges in $F_a^*$ results in a graph in $\mathcal G$.
Since every edge $e\in F_a^*$ is a virtual edge of $G_a$, the two endpoints of $e$ are non-adjacent in $G$.
Thus, the edges in $F_a^*$ correspond to a collection of internally vertex-disjoint paths $\{P_e:e\in F_a^*\}$ such that each $P_e$ is an induced path of length at least two in $H'_a$ (and thus in $G$) connecting the endpoints of $e$.
Consider an arbitrary set $F\subseteq F_a^*$.
Since $G_a'$ is an induced subgraph of $G$, each real edge in $G_a$ is also an edge in $G_a'$.
It follows that the graph $G_a-F$ can be obtained from $G_a'$ by deleting the internal vertices of $P_e$ for all $e\in F$ and by contracting to a single edge each path $P_e$ for all $e\in F_a^*\setminus F$.
Hence, $G_a-F$ is an induced topological minor of $G_a'$ and thus of $G$.
Since $G\in \mathcal{G}$ and $\mathcal{G}$ is closed under induced topological minors, we have $G_a-F\in \mathcal{G}$, as claimed.
\end{subproof}

\medskip
\noindent{\bf Step 2.} We compute an $\ell$-refined tree decomposition $\widehat{\mathcal{T}}_a = (T_a, \{(X^a_t,U^a_t) : t \in V(T_a)\})$ of $G[X_a]$ with residual independence number at most $\max\{3-\ell,k\}$ and an $F_a$-mapping of $\widehat{\mathcal{T}}_a$.
We do so by computing an $\ell$-refined tree decomposition of $G_a$.
Since $G[X_a]$ is a subgraph of $G_a$, the decomposition $\widehat{\mathcal{T}}_a$ is also an $\ell$-refined tree decomposition of $G[X_a]$.
We consider two cases depending on whether $a$ is an \textnormal{S}-node or an \textnormal{R}-node of $S$.

\begin{itemize}
\item If $a$ is an \textnormal{S}-node of $S$, then $G_a$ is a cycle and we can apply a similar approach as in \cref{tin-cycles}.
In time $\mathcal{O}(|V(G_a)|)$ we compute a cyclic order $v_1,\ldots, v_h$ of the vertices of the cycle.
We construct the desired $\ell$-refined tree decomposition $\widehat{\mathcal{T}}_a$ of $G_a$ as follows.
The tree $T_a$ is an $(h-2)$-vertex path $(t_1,\ldots, t_{h-2})$; for each $i\in \{1,\ldots, h-2\}$, the bag $X^a_{t_i}$ consists of vertices $\{v_i,v_{i+1},v_h\}$, and $U^a_{t_i}$ is any subset of $X^a_{t_i}$ with exactly $\min \{\ell,|X^a_{t_i}| \}$ vertices.
Note that every bag $X^a_{t}$ of $\widehat{\mathcal{T}}_a$ has size $3$ and hence
\[\alpha(G[X^a_{t} \setminus U^a_{t}]) \le |X^a_{t}|-|U^a_{t}|\le \max \{3 - \ell,0\}\le \max\{3-\ell,k\}\,.\]
The corresponding $F_a$-mapping of $\widehat{\mathcal{T}}_a$ can be obtained as follows.
For each edge $e = v_iv_{i+1} \in F_a$ with $1\le i\le h$ (indices modulo $h$), we do the following:
if $1\le i\le h - 2$, we map $e$ to $t_i\in V(T_a)$;
if $i = h - 1$, we map $e$ to $t_{h-2}\in V(T_a)$;
and if $i = h$, we map $e$ to $t_{1}\in V(T_a)$.
Clearly, both $\widehat{\mathcal{T}}_a$ and the $F_a$-mapping can be obtained in $\O(|V(G_a)|)$ time.

\item Otherwise, $a$ is an \textnormal{R}-node of $S$.
Recall that $\mathcal G$ is closed under induced topological minors.
By \cref{G_a-induced-topological-minors}, $G_a$ is an induced topological minor of $G\in \mathcal G$, and thus $G_a$ belongs to $\mathcal G$.
By \cref{claim:virtually-safe}, $F_a^*$ is a $\mathcal G$-safe set of edges of $G_a$.
Since $G_a$ is $3$-connected, by the assumption of the lemma, we can compute an $F^*_a$-covering $\ell$-refined tree decomposition $\widehat{\mathcal{T}}_a$ of $G_a$ with residual independence number at most $k$ and an $F_a$-mapping of $\widehat{\mathcal{T}}_a$ in time $f(|V(G_a)|,|E(G_a)|)$.
Let $t$ be a node of $\widehat{\mathcal{T}}_a$.
Since $U^a_t$ is an $F_a^*\!$-cover of $G_a[X^a_t]$, every virtual edge in $G_a[X^a_t]$ has an endpoint in $U^a_t$.
We infer that the subgraphs of $G$ and $G_a$ induced by the set $X^a_t\setminus U^a_t$ are the same and thus, \[\alpha(G[X^a_t\setminus U^a_t]) = \alpha(G_a[X^a_t\setminus U^a_t])\le \widehat{\alpha}(\widehat{\mathcal{T}}_a)\le k\le \max\{3-\ell,k\}\,.\]
\end{itemize}

\medskip
\noindent{\bf Step 3.} This step consists in, informally speaking, replacing the node $a$ and the bag $X_a$ in $\widehat{\mathcal{T}}$ by the newly computed $\ell$-refined tree decomposition $\widehat{\mathcal{T}}_a$ of $G_a$.
We first compute the forest $T' =(T-a) + T_a$ and then make $T'$ connected by iterating over all neighbors $b$ of $a$ in $S$.
Since every edge of $S$ has exactly one endpoint which is a \textnormal{P}-node and $a$ is not a \textnormal{P}-node of $S$, we infer that $b$ must be a \textnormal{P}-node of $S$, and thus $X_b$ contains exactly two vertices in $G$, which are adjacent in $G_{a}$.
Let $e \in E(G_a)$ be the edge with endpoints in $X_b$.
By the definition of $F_a$, the edge $e$ belongs to $F_a$.
Recall that an $F_a$-mapping of $\widehat{\mathcal{T}}_a$ has been computed in Step~2.
Hence, given the edge $e$, the $F_a$-mapping of $\widehat{\mathcal{T}}_a$ returns, in constant time, a node $c$ of the tree $T_a$ whose bag $X_c$ contains the endpoints of $e$.
We connect $b$ to $c$ in $T'$.
Notice that once all the neighbors of $a$ have been considered, $T'$ becomes a tree.
We now set $\widehat{\mathcal{T}} = (T', \{(X_t,U_t) : t \in V(T-a)\} \cup \{(X^a_t,U^a_t) : u \in V(T_a)\})$.
It is not difficult to verify that after this modification, $\widehat{\mathcal{T}}$ remains an $\ell$-refined tree decomposition of $G$.
The time complexity of this step is proportional to $\mathcal{O}(d_S(a) + |\widehat{\mathcal{T}}_a|)$.

\medskip
We now reason about the overall time complexity of Steps 1--3 for a fixed \textnormal{R}-node or \textnormal{S}-node $a$ of $S$.
Recall that $N_{\text{R}}$ and $N_\text{S}$ denote the sets of \textnormal{R}-nodes and \textnormal{S}-nodes of $S$, respectively.
Assume first that $a\in N_{\text{S}}$.
Then the complexity is $\mathcal{O}(d_S(a) + |V(G_a)|+|\widehat{\mathcal{T}}_a|)$.
Since $G_a$ is a cycle and our construction of $\widehat{\mathcal{T}}_a$ implies that $|\widehat{\mathcal{T}}_a| = \O(|V(G_a)|)$,
this simplifies to $\mathcal{O}(d_S(a) + |V(G_a)|)$.
Assume now that $a\in N_{\text{R}}$.
Then the complexity is $\mathcal{O}(d_S(a) + f(|V(G_a)|,|E(G_a)|)+ |\widehat{\mathcal{T}}_a|)$,
which simplifies to $\mathcal{O}(d_S(a) + f(|V(G_a)|,|E(G_a)|))$ since $f(|V(G_a)|,|E(G_a)|)\ge |\widehat{\mathcal{T}}_a|$.

\medskip
By construction, the final $\ell$-refined tree decomposition $\widehat{\mathcal{T}}$ has residual independence number at most $\max\{3-\ell,k\}$.
The overall time complexity of the algorithm is
\[\O\left(n+m +\sum_{a\in N_{\text{S}}}(d_S(a) + |V(G_a)|) + \sum_{a\in N_{\text{R}}}(d_S(a) + f(|V(G_a)|,|E(G_a)|))\right)\,,\]
or, equivalently,
\[
    \mathcal{O}\left(n+m +
\sum_{a\in N_{\text{R}}\cup N_{\text{S}}}d_S(a)
+ \sum_{a\in N_{\text{S}}}|V(G_a)| + \sum_{a\in N_{\text{R}}}f(|V(G_a)|,|E(G_a)|)\right)\,.
\]
\Cref{SPQR-sum-of-bags} implies that $\sum_{a\in N_{\text{S}}}|V(G_a)|\le \sum_{a\in N_{\text{R}}\cup N_{\text{S}}}|X_a|  \le 3n-6$.
Furthermore, we know that $\sum_{a\in N_{\text{R}}\cup N_{\text{S}}}d_S(a) = |E(S)| = \O(m)$~\cite{DBLP:conf/gd/GutwengerM00}.
Since the function $f$ is superadditive, we have
\[\sum_{a\in N_{\text{R}}}f(|V(G_a)|,|E(G_a)|) \leq f\left( \sum_{a\in N_{\text{R}}} |V(G_a)|, \sum_{a\in N_{\text{R}}} |E(G_a)|\right)\,.\]
By \cref{SPQR-sum-of-bags}, we have that $\sum_{a\in N_{\text{R}}} |V(G_a)| \leq 2n$.
Additionally, by \cref{SPQR-sum-of-edges} we get that $\sum_{a\in N_{\text{R}}} |E(G_a)| \leq 3m$.
Since $f$ is non-decreasing in each coordinate and $m\ge n$ as $G$ is $2$-connected, the running time simplifies to $\mathcal{O}\left(m + f(2n,3m)\right)$.

The number of nodes of $\widehat{\mathcal{T}}$ is at most
\[|N_{\text{P}}|+\sum_{a\in N_{\text{S}}} (|V(G_a)|-2) + \sum_{a\in N_{\text{R}}} g(|V(G_a)|,|E(G_a)|)\,.\]
Following the fact that $|N_{\text{P}}| \le |V(S)| = \mathcal{O}(n)$ and by \cref{SPQR-sum-of-bags}, we get that $|N_{\text{P}}|+\sum_{a\in N_{\text{S}}} (|V(G_a)|-2) = \mathcal{O}(n)$.
Applying the same arguments as we did for the function $f$ to the function $g$ shows that $\sum_{a\in N_{\text{R}}} g(|V(G_a)|,|E(G_a)|)\le g(2n,3m)$.
Thus, the number of nodes of $\widehat{\mathcal{T}}$ is of the order $\mathcal{O}(n+g(2n,3m))$.
This completes the proof.
\end{proof}

We now combine the reduction using block-cutpoint trees (given by  \cref{cor:reduction-to-blocks}) with the reduction using SPQR trees (given by \cref{lem:sufficient}) to derive the main result of this section.

\begin{sloppypar}
\begin{theorem}\label{thm:sufficient}
Let $\mathcal{G}$ be a graph class closed under induced topological minors for which there exist nonnegative integers $k$ and $\ell$ such that $\mathcal{G}$ is $(k,\ell)$-tree decomposable and {\sc $(k,\ell)$-tree Decomposition($ \mathcal{G})$} can be solved in time $f(n,m)$ on graphs with $n$ vertices and $m$ edges so that the resulting $\ell$-refined tree decomposition has $g(n,m)$ nodes, where $f$ and $g$ are superadditive functions.
Then, for any graph $G$ in $\mathcal{G}$ with $n\ge 1$ vertices and $m$ edges, one can compute in time $\mathcal{O}(n + m + f(4n,3m))$ an $\ell$-refined tree decomposition of $G$ with $\mathcal{O}(n + g(4n,3m))$ nodes and with residual independence number at most $\max\{3 - \ell,k\}$.
\end{theorem}
\end{sloppypar}

\begin{proof}
By \cref{lem:sufficient}, there exist positive integers $c$ and $d$ such that for any $2$-connected graph $G \in \mathcal{G}$ with $n$ vertices and $m$ edges, one can compute in time $c \cdot (m + f(2n,3m))$ an $\ell$-refined tree decomposition of $G$ with residual independence number at most $\max\{3 - \ell,k\}$ and with $d \cdot (n + g(2n,3m))$ nodes.
Let $\hat k = \max\{3 - \ell,k\}$ and let us define two functions, $\hat{f}: \mathbb{Z}_+\times \mathbb{Z}_+\to \mathbb{Z}_+$ and $\hat{g}: \mathbb{Z}_+\times \mathbb{Z}_+\to \mathbb{Z}_+$, as follows:
for each $(x,y)\in \mathbb{Z}_+\times \mathbb{Z}_+$, we set $\hat{f}(x,y) = c \cdot (y + f(2x,3y))$ and $\hat{g}(x,y) = d \cdot (x + g(2x,3y))$.
Thus, $\mathcal{G}$ is a hereditary graph class such that for each $2$-connected graph $G$ in ${\mathcal G}$ with $n$ vertices and $m$ edges, one can compute in time $\hat{f}(n,m)$ an $\ell$-refined tree decomposition $\widehat{\mathcal{T}}$ with $\hat{g}(n,m)$ nodes and with residual independence number at most $\hat k$.
Furthermore, the fact that $f$ and $g$ are superadditive implies that $\hat{f}$ and $\hat{g}$ are superadditive.
By \cref{cor:reduction-to-blocks}, for any graph $G$ in $\mathcal{G}$ with $n\ge 1$ vertices and $m$ edges, one can compute in time $\mathcal{O}(n + m + \hat{f}(2n,m))$ an $\ell$-refined tree decomposition of $G$ with $\mathcal{O}(n + \hat{g}(2n,m))$ nodes and with residual independence number at most $\max\{2-\ell, \hat k\} = \max\{3-\ell,k\}$.
Since $\hat{f}(2n,m) = c \cdot (m + f(4n,3m))$ and $\hat{g}(2n,m) = c \cdot (n + g(4n,3m))$, the theorem follows.
\end{proof}

\section{$W_4$-induced-minor-free graphs}\label{W4-induced-minor-free}

\begin{sloppypar}
In this section and the next one, we apply \cref{thm:sufficient} to prove that if $\mathcal{G}$ is the class of \hbox{$H$-induced-minor-free graphs} for $H\in \{W_4,K_5^-\}$, then $\mathcal{G}$ has bounded and efficiently witnessed tree-independence number.
In fact, we show that $\mathcal{G}$ is $(1,3)$-tree decomposable and develop a polynomial-time algorithm for the {\sc $(1,3)$-tree Decomposition($\mathcal{G})$} problem.
Combining this result with \cref{thm:bounded-ell-refined-tree-independence-number} leads to an $\mathcal{O}(|V(G)|^3)$ algorithm for the \textsc{Max Weight Independent Set} problem for vertex-weighted graphs $G\in \mathcal{G}$.
Similarly, applying \cref{thm:bounded-tree-independence-number-packings} leads to polynomial-time algorithms for the \textsc{Max Weight Independent Packing} problem.

Since the {\sc $(k,\ell)$-tree Decomposition($\mathcal{G})$} problem deals with $3$-connected graphs, we first need to characterize the $3$-connected graphs in $\mathcal{G}$.
We start with the $W_4$-induced-minor-free graphs.
\end{sloppypar}

\medskip
\begin{figure}[ht]
  \centering
\begin{tikzpicture}[scale=1]
\tikzset{every node/.style={draw,circle,fill=black,inner sep=0pt,minimum size=4.5pt}}

    \clip(-3,-3) rectangle (5.5,3.5);

    \draw[ultra thick,red] ([shift=(100:2.2)]0,0) arc (100:260:2.2cm);
    \coordinate[label={[text=red]left:$\strut P$}] (P) at (200:2.2);
    \draw[ultra thick,black!20!green] ([shift=(-80:2.2)]0,0) arc (-80:80:2.2cm);
    \coordinate[label={[text=black!20!green]right:$\strut Q$}] (Q) at (-20:2.2);

    \draw (0,0) circle(2);
    \node[label=above:$\strut u$] (u) at (0,2) {};
    \node[label=below:$\strut v$] (v) at (0,-2) {};
    \node (x) at (180:2) {};
    \coordinate (x') at (185:2) {};
    \node (y) at (20:2) {};
    \coordinate (y') at (15:2) {};
    \node[label=below left:$\strut w$] (w) at (0:2) {};

    \draw[ultra thick, blue, shorten <= 0.2cm, shorten >= 0.2cm] (x') to[bend right=60] node[midway,label={[text=blue]below:$\strut R$},draw=none,fill=none,minimum size=0pt] {} (y');
    \draw (x) to[bend right=60] (y);

\makeatletter
\pgfdeclareshape{bean}{
\anchor{center}{\pgfpoint{0}{0cm}}%

\backgroundpath{ %
    \pgfsetlinewidth{\pgfkeysvalueof{/pgf/minimum width}*0.035}
    \path[draw, clip] plot[smooth, tension=1] coordinates {(-3.5,0.5) (-3,2.5) (-1,3.5) (1.5,3) (4,3.5) (5,2.5) (5,0.5) (2.5,-2) (0,-0.5) (-3,-2) (-3.5,0.5)};
    \fill[black!5!white] (-4,-4) rectangle (6, 4,);
}%
}
\makeatother

    \coordinate (H) at (4,1);
    \draw (H) to[out=120, in=50, looseness=1.5] (u)
          (H) to[bend left=10] (w)
          (H) to[in=-60,out=-80,looseness=1] (v);
    \draw (u) to[out=60, in=110,looseness=1.5] ($(H)+(0,.3)$);
    \draw (w) to[bend right=30] ($(H)+(0,-.2)$);

    \node[rotate=-90,scale=.3,bean,minimum size=50pt, label=below right:$\strut H$] at (H) {};
\end{tikzpicture}
  \caption{A schematic representation of the situation in the proof of \Cref{lem:3-con-W_4-IM-free}.}\label{fig:W4}
\end{figure}
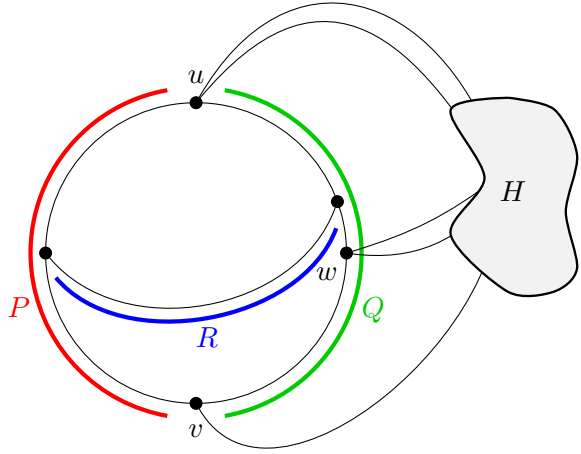

\begin{lemma}\label{lem:3-con-W_4-IM-free}
Let $G$ be a $3$-connected graph.
Then $G$ is $W_4$-induced-minor-free if and only if $G$ is chordal.
\end{lemma}

\begin{proof}
First, let us assume that $G$ is a $3$-connected $W_4$-induced-minor-free graph.
Suppose that there exists an induced cycle $C$ of length at least $4$ in $G$.
Since $G$ is $3$-connected, it cannot equal to $C$.
Thus, there exists a component $H$ of $G-V(C)$.
Let $S$ be the set of vertices of $C$ that have a neighbor in $H$.
The $3$-connectedness of $G$ implies that $|S|\ge 3$.
On the other hand, $|S|\le 3$, since otherwise we could contract $H$ into a single vertex, delete all the other components, and contract some edges of $C$ to obtain a $W_4$ as an induced minor of $G$.
Thus, $|S| = 3$.
Since $C$ has length at least $4$, there exists a pair $\{u,v\}$ of nonadjacent vertices in $S$.
Let $P$ and $Q$ be the two paths of the graph $C-\{u,v\}$.
Note that $P$ and $Q$ are both nonempty.
Furthermore, one of the two paths $P$ and $Q$ does not contain any vertex from $S$.
By symmetry, we may assume that $V(P)\cap S = \emptyset$, that is, no vertex from $H$ has a neighbor in $P$.
It follows that the third neighbor of $H$ on $C$, call it $w$, belongs to $V(Q)$.
Since $G$ is $3$-connected, the graph $G-\{u,v\}$ is connected, and hence contains a path from $V(P)$ to $V(Q)$.
Let $R'$ be a shortest path in $G-\{u,v\}$ from $V(P)$ to $V(Q)$ and let $R$ be the subpath obtained by deleting the endpoints of $R'$.
Note that since $C$ is an induced cycle, the path $R$ is nonempty.
Furthermore, since $w$ is the only vertex in $Q$ with a neighbor in $H$, the path $R$ cannot contain any vertex of $H$ (see \Cref{fig:W4}).
We now derive a contradiction by showing that there exists an induced minor model of $W_4$ in $G$.
Let $X_0 = V(Q)\cup V(R)$, $X_1 = \{u\}$, $X_2 = V(P)$, $X_3 = \{v\}$, and $X_4 = V(H)$.
Then $X_0,X_1,\ldots, X_4$ are pairwise disjoint subsets of $V(G)$, each inducing a connected subgraph, such that for each $i\in \{1,2,3,4\}$ there is an edge in $G$ with one endpoint in $X_0$ and the other in $X_i$ as well as an edge in $G$ with one endpoint in $X_i$ and the other in $X_{i+1}$ (with $X_5 = X_1$), no edge in $G$ has one endpoint in $X_1$ and another in $X_3$, and no edge in $G$ has one endpoint in $X_2$ and another in $X_4$.
Therefore, the sets $X_0,X_1,\ldots, X_4$ form an induced minor model of $W_4$ in $G$, a contradiction.

For the other direction, assume that $G$ is a chordal graph.
Suppose that $G$ contains $W_4$ as an induced minor.
It follows that $G$ also contains $C_4$ as an induced minor.
But then, $G$ contains some cycle of length at least four as an induced subgraph which is a contradiction with the assumption that $G$ is chordal.
\end{proof}

By \cref{lem:3-con-W_4-IM-free}, every $3$-connected graph in the class  $\mathcal{G}$ of $W_4$-induced-minor-free graphs is chordal.
To show that $\mathcal{G}$ is $(1,3)$-tree decomposable and develop a linear-time algorithm for the {\sc $(1,3)$-tree Decomposition($\mathcal{G})$} problem (\Cref{w4-(124)}), we need the following concepts related to chordal graphs.

Let $G$ be a graph.
A \emph{vertex ordering} of $G$ is a total order $(v_1,\ldots, v_n)$ of its vertices.
A \emph{module} in $G$ is a set $M\subseteq V(G)$ such that every vertex not in $M$ that has a neighbor in $M$ is adjacent to all vertices in $M$.
A \emph{moplex} in $G$ is an inclusion-maximal module $M\subseteq V(G)$ that is a clique and its neighborhood $N(M)$ is either empty or a minimal separator in $G$.
A \emph{perfect moplex partition} of $G$ is an ordered partition $(M_1,\ldots, M_k)$ of $V(G)$ such that for all $i\in \{1,\ldots, k\}$, $M_i$ is a moplex in the subgraph of $G$ induced by $\cup_{j = i}^k M_j$.
Given an ordered partition $\pi = (Z_1,\ldots, Z_k)$ of $V(G)$ and a vertex ordering $\sigma = (v_1,\ldots, v_n)$ of $G$, we say that $\sigma$ is \emph{compatible with} $\pi$ if $v_i\in Z_p$, $v_j\in Z_q$, and $p<q$ imply $i<j$.
Berry and Bordat showed in~\cite{MR1626534} that every graph has a moplex.
Consequently, every graph has a perfect moplex partition.
A \emph{perfect moplex ordering} of $G$ is a vertex ordering compatible with a perfect moplex partition.
Using graph search algorithms such as Lexicographic Breadth-First Search (LexBFS)~\cite{MR408312} or Maximum Cardinality Search (MCS)~\cite{MR749707}, one can compute in linear time a perfect moplex ordering of a given graph by reversing the ordering returned by LexBFS or MCS (see~\cite{MR1626534,MR1320296,MR2659379}).
The reader may know the concept of \emph{perfect elimination ordering}; it is easily observed that every perfect moplex ordering is a perfect elimination ordering.

\begin{sloppypar}
\begin{lemma}\label{chordal-F-mapping}
Let $G$ be a chordal graph with $n$ vertices and $m$ edges and $F\subseteq E(G)$.
Then one can compute in time $\mathcal{O}(n+m)$ a clique tree $\mathcal T$ of $G$ with at most $n$ nodes and an $F$-mapping of $\mathcal T$.
\end{lemma}
\end{sloppypar}

\begin{proof}
Using LexBFS or MCS, we compute in linear time a perfect moplex ordering  $(v_1,\ldots, v_n)$ of $G$.
Berry and Simonet gave in~\cite{MR3632036} a linear-time algorithm that takes as input a connected chordal graph $G$ and a perfect moplex ordering of $G$, and computes a clique tree ${\mathcal T}$ of $G$.
We explain the idea of their algorithm in terms of the perfect moplex partition $(M_1,\ldots, M_k)$ of $G$ corresponding to the given perfect moplex ordering.
The bags of the computed clique tree $\mathcal{T}$ are the maximal cliques of $G$, which are exactly the sets $X_1,\ldots, X_k$ of the form $X_i = N_{G_i}[M_i]$ for all $i\in \{1,\ldots, k\}$ where $G_i$ is the subgraph of $G$ induced by $\cup_{j = i}^k M_j$.
The algorithm processes the moplexes in order from $M_k$ to $M_1$.
It starts with a clique tree $\mathcal{T}_k$ of $G_k$ with a tree containing a unique node $k$ and the corresponding bag $X_k = M_k$.
Then, for each $i= k-1,\ldots, 1$, the algorithm computes a clique tree $\mathcal{T}_i$ of $G_i$ from the clique tree $\mathcal{T}_{i+1}$ of $G_{i+1}$ by adding to the tree of ${\mathcal T}_{i+1}$ a new node $i$ associated with bag $X_i$, and an edge $(i,j)$ where $j$ is the smallest number in $\{i+1,\ldots, k\}$ such that there is an edge in $G$ from $M_i$ to $M_j$.
The final clique tree of $G = G_1$ is given by ${\mathcal T} = {\mathcal T}_1$.

We compute an $F$-mapping of ${\mathcal T}$ as follows.
For each edge $v_iv_j\in F$ such that $i<j$, we assign the edge $v_iv_j$ to the unique node $p\in \{1,\ldots, k\}$ of $\mathcal{T}$ such that $v_i\in M_p$.
This can be done in time $\mathcal{O}(|F|) = \mathcal{O}(|E(G)|)$.

Let us justify that the so-defined mapping is indeed an $F$-mapping of ${\mathcal T}$.
Consider an edge $v_iv_j\in F$ with $i<j$, and let $p$ and $q$ be the unique nodes of $\mathcal{T}$ such that $v_i\in M_p$ and $v_j\in M_q$, respectively.
Since the vertex ordering $(v_1,\ldots, v_n)$ is compatible with the perfect moplex partition $(M_1,\ldots, M_k)$, we have $p \le q$.
We need to show that $v_i$ and $v_j$ both belong to $X_p = N_{G_p}[M_p]$.
First, we have that $v_i\in M_p\subseteq X_p$.
Second, since $p\le q$, we have $v_j\in V(G_p)$ and thus $v_j$ is adjacent to $v_i$ in $G_p$. In particular, we have $v_j\in N_{G_p}[M_p]$. Thus, $v_iv_j\subseteq X_p$, which is what we wanted to show.

Let now $G$ be an arbitrary chordal graph with $n$ vertices and $m$ edges.
The above approach can be extended in a straightforward way to the case when $G$ is not connected, by computing the connected components and applying the above algorithm to each component.
The resulting forest of clique trees of the components can be turned into a clique tree of $G$ by adding the appropriate number of edges between the clique trees of the components.
By construction, the obtained clique tree ${\mathcal T}$ of $G$ has at most $n$ nodes.
\end{proof}

\begin{lemma}\label{w4-(124)}
Let $\mathcal{G}$ be the class of $W_4$-induced-minor-free graphs.
Then $\mathcal{G}$ is $(1,3)$-tree decomposable and {\sc $(1,3)$-Tree Decomposition($\mathcal{G}$)} can be solved on graphs with $n$ vertices and $m$ edges in time $\mathcal{O}(n \cdot m)$
so that the resulting $3$-refined tree decomposition with residual independence number at most $1$ has at most $n$ nodes.
\end{lemma}

\begin{proof}
Let $G$ be a $3$-connected $W_4$-induced-minor-free graph with $n$ vertices and $m$ edges and let $F^*\subseteq F\subseteq E(G)$ such that $F^*$ is $\mathcal{G}$-safe.
We want to compute an $F^*\!$-covering $3$-refined tree decomposition $\widehat{\mathcal{T}}$ of $G$ with $\widehat{\alpha}(\widehat{\mathcal{T}})\le 1$, together with an $F$-mapping of $\widehat{\mathcal{T}}$.

By \cref{lem:3-con-W_4-IM-free}, $G$ is chordal.
By \cref{chordal-F-mapping}, one can compute in linear time a $0$-refined tree decomposition $\widehat{\mathcal{T}}= (T,\{(X_t,U_t)\}_{t\in V(T)})$ of $G$ with residual independence number $1$ having at most $n$ nodes and an $F$-mapping of $\widehat{\mathcal{T}}$.
Note that $\widehat{\mathcal{T}}$ already satisfies all the desired properties, except that it may fail to be $F^*\!$-covering.
To fix this, we next redefine the sets $U_t$ for all $t\in V(T)$.

Let $t \in V(T)$.
If $|X_t| \leq 4$, then we set $U_t$ to be any subset of $\min\{|X_t|,3\}$ vertices of $X_t$.
This in particular implies that $|X_t \setminus U_t| \leq 1$, and hence $U_t$ is an $F^*\!$-cover of $G[X_t]$ of size at most $3$.

Now suppose that $|X_t| \geq 5$.
Consider the set $L$ of all edges in $F^*$ that have both endpoints in $X_t$.
We claim that any two edges in $L$ have a common endpoint.
Suppose that this is not the case, and let $e = uv$ and $f = xy$ be two disjoint edges in $L$.
As $|X_t|\ge 5$, there exists a vertex $z$ in $X_t$ distinct from any of $u,v,x,y$.
Since $X_t$ is a clique in $G$, the subgraph $H$ of $G$ induced by $\{u,v,x,y,z\}$ is isomorphic to $K_5$, and thus $H-\{e,f\}$ is isomorphic to $W_4$.
This implies that $G-\{e,f\}$ is not $W_4$-induced-minor-free.
However, since $\{e,f\}\subseteq L\subseteq F^*$, this is a contradiction with the fact that $F^*$ is $\mathcal{G}$-safe.

If $L = \emptyset$, we set $U_t = \emptyset$.
Otherwise, we choose an arbitrary edge $e =xy\in L$ and set $U_t = \{x,y\}$.
Such a set $U_t$ can be computed in time $\mathcal{O}(|X_t| \cdot |F^*|) = \mathcal{O}(n \cdot m)$, by iterating over the edges in $F^*$ and checking whether their endpoints both belong to $X_t$.
Note that since any two edges in $L$ have a common endpoint, all edges of the graph $G[X_t\setminus U_t]$ are in $E(G) \setminus F^*$, and thus $U_t$ is an $F^*\!$-cover of $G[X_t]$ of size at most $2$.
\end{proof}

\Cref{thm:sufficient,w4-(124)} imply the following.

\begin{theorem}\label{W_4-tree-independence}
For any $W_4$-induced-minor-free graph $G$ with $n$ vertices and $m$ edges, one can compute in time $\mathcal{O}(n \cdot m)$ a $3$-refined tree decomposition of $G$ with $\mathcal{O}(n)$ nodes and residual independence number at most $1$.
\end{theorem}

\begin{proof}
The class of $W_4$-induced-minor-free graphs is closed under induced topological minors, and \cref{w4-(124)} shows that \cref{thm:sufficient} is satisfied with $k = 1$ and $\ell = 3$, $f(n,m) = \mathcal{O}(n \cdot m)$, and $g(n,m) = n$.
Therefore, given a $W_4$-induced-minor-free graph $G$ with $n$ vertices and $m$ edges, \cref{thm:sufficient} implies that one can compute an $\ell$-refined tree decomposition of $G$ with $\mathcal{O}(n)$ nodes and with residual independence number at most $\max \{3-\ell,k\} = 1$ in time $\mathcal{O}(n+m+f(4n,3m))$, which is $\mathcal{O}(n \cdot m)$.
\end{proof}

\begin{remark}\label{W4-robust}
The algorithm for computing a $3$-refined tree decomposition of a $W_4$-induced-minor-free graph can be turned into a polynomial-time robust algorithm, as follows.
First, after we compute the SPQR tree $S$ of a given graph $G$, we check for every $R$-node $a$ if the corresponding graph $G_a$ in $S$ is a chordal graph.
If this is not the case, the algorithm returns that the graph does not belong to the class of $W_4$-induced-minor-free graphs.
This check can be done in linear time.
Assume now that each graph $G_a$ corresponding to an $R$ node $a$ of $S$ is chordal.
Then, once we compute a clique tree $\mathcal{T}$ of $G_a$, we check for each bag $X$ of $\mathcal{T}$ if the graph $G[X]$ contains an independent set of size $5$, which can be checked in polynomial time.
If some graph $G[X]$ contains an independent set of size $5$, then the algorithm returns that $G$ is not in the class of $W_4$-induced-minor-free graphs; otherwise, the algorithm correctly computes a $3$-refined tree decomposition of $G$ with residual independence number at most $1$.\hfill$\blacktriangle$
\end{remark}

In the first paper of the series~\cite{dallard2021treewidth} we established $(\tw,\omega)$-boundedness of the class of $W_4$-induced-minor-free graphs, although without an explicit binding function.
\Cref{bounded tin implies bounded tw-omega-refined,W_4-tree-independence} imply that the class of $W_4$-induced-minor-free graphs admits a linear $(\tw,\omega)$-binding function $f(p) = R(p+1,2)+3-2 = p+2$.
Furthermore, our approach leading to the proof of \cref{W_4-tree-independence} can be used to show, additionally, that if $G$ is a $W_4$-induced-minor-free graph, then the inequality $\tw(G) \ge  \eta(G)-1$, which holds for all graphs, is satisfied with equality.
Indeed, assuming that $G$ is $2$-connected, the proof of \cref{W_4-tree-independence} provides a tree decomposition of $G$ obtained from tree decompositions of graphs $G_a$ over all \textnormal{R}- and \textnormal{S}-nodes $a$ of a fixed SPQR tree of $G$.
Using the structure of the corresponding graphs $G_a$, it can be verified that the constructed tree decompositions are only composed of bags with at most $\eta(G_a)$ vertices.
Applying \cref{corollary-eta} then shows that the overall tree decomposition only contains bags with at most $\eta(G)$ vertices.

\begin{proposition}
If $G$ is a $W_4$-induced-minor-free graph, then $\tw(G) = \eta(G)-1$.
\end{proposition}

Let us emphasize that for general classes of $H$-induced-minor-free graphs where $H$ is planar, only the existence of a function bounding the treewidth in terms of the Hadwiger number is known (see~\cite{MR3741534,MR3853110,MR2901091}).

By \cref{W_4-tree-independence}, every $W_4$-induced-minor-free graph $G$ satisfies the inequality $3$-$\tin(G)\le 1$.
Using \cref{observation-ell-tree-alpha} we thus obtain the following.

\begin{corollary}\label{W_4-tree-independence-number}
The tree-independence number of any $W_4$-induced-minor-free graph is at most $4$.
\end{corollary}

\begin{remark}\label{remark:sharp}
The bound on the tree-independence number given by \cref{W_4-tree-independence-number} is sharp: there exist arbitrarily large $2$-connected $W_4$-induced minor-free graphs with tree-independence number~$4$.
Take an integer $q\ge 4$ and let $F_q$ be the graph obtained from a complete graph with vertex set $S = \{1,2,3,4\}$ by replacing each of its edges $ij$ with $q$ paths of length two connecting $i$ and $j$.
Note that $F_q$ has exactly $4$ vertices of degree more than two.
Neither deleting vertices nor contracting edges having an endpoint of degree at most two can increase the number of vertices of degree more than two.
It follows that every induced minor of $F_q$ has at most $4$ vertices of degree more than two.
Since the graph $W_4$ has $5$ vertices and minimum degree $3$, we infer that $F_q$ is $W_4$-induced-minor-free, and hence $\tin(F_q) \le 4$ by \cref{W_4-tree-independence}.
To see that the inequality is satisfied with equality, it suffices to show that $\tin(F_4)\ge 4$.
But this was already observed in the proof of~\cite[Theorem 3.8]{dallard2022firstpaper}.
\hfill$\blacktriangle$
\end{remark}

\cref{W_4-tree-independence,thm:bounded-ell-refined-tree-independence-number} imply the existence of a polynomial-time algorithm for the \textsc{Max Weight Independent Set} problem in the class of $W_4$-induced-minor-free graphs.

\begin{corollary}\label{MWIS for W4-im-free graphs}
The \textsc{Max Weight Independent Set} problem can be solved in time $\mathcal{O}(n^{3})$ for $n$-vertex $W_4$-induced-minor-free graphs.
\end{corollary}

\begin{remark}\label{refined-improvement-1}
By applying \cref{thm:bounded-tree-independence-number} instead of \cref{thm:bounded-ell-refined-tree-independence-number}, one could also use usual tree decompositions instead of $\ell$-refined ones.
However, this approach would only lead to a running time of $\mathcal{O}(n^{6})$ instead of $\mathcal{O}(n^{3})$.
\hfill$\blacktriangle$
\end{remark}

Furthermore, since a $3$-refined tree decomposition with residual independence number at most $1$ has independence number at most $4$, \cref{W_4-tree-independence,thm:bounded-tree-independence-number-packings} imply the existence of a polynomial-time algorithm for the \textsc{Max Weight Independent Packing} problems in the class of $W_4$-induced-minor-free graphs.

\begin{theorem}\label{MWIHP for W4-im-free graphs}
Given a $W_4$-induced-minor-free graph $G$ and a finite family $\mathcal{H} = \{H_j\}_{j \in J}$ of connected nonnull subgraphs of $G$, the \textsc{Max Weight Independent Packing} problem can be solved in time \hbox{$\mathcal{O}(|J|\cdot |V(G)|\cdot (|V(G)| + |J|^4))$}.
\end{theorem}

\begin{proof}
Let $G$ be an $n$-vertex $W_4$-induced-minor-free graph, given along with a finite family $\HH = \{H_j\}_{j \in J}$ of connected nonnull subgraphs and a weight function \hbox{$w:J\to \mathbb{Q}_+$}.
By \cref{W_4-tree-independence,observation-ell-tree-alpha}, we can compute in time $\mathcal{O}(n^3)$ a tree decomposition \hbox{$\mathcal{T} = (T, \{\Bag_t\}_{t\in V(T)})$} of $G$ with $\mathcal{O}(n)$ nodes and with independence number at most $4$.
The conclusion now follows from \cref{thm:bounded-tree-independence-number-packings}.
\end{proof}

\begin{remark}\label{remark:cant-apply-MWIS-directly}
Cameron and Hell showed in~\cite{MR2190818} that if $G$ is a chordal graph, then for every nonempty family $\FF$ of connected graphs, the derived graph $G(\HH)$ where $\HH = \HH(G,\FF)$ is chordal.
One may wonder whether a similar statement holds for the more general class of $W_4$-induced-minor-free graphs.
This would be algorithmically interesting as it would imply the existence of an $\mathcal{O}(n^{3r})$ algorithm for the \textsc{Max Weight Independent $\FF$-Packing} problem for $n$-vertex $W_4$-induced-minor-free graphs, where $r$ is the maximum order of a graph in $\FF$, a significant improvement of the $\mathcal{O}(n^{5r+1})$ time complexity following from \cref{MWIHP for W4-im-free graphs}.\footnote{We could compute in time $\mathcal{O}(n^{2r})$ the graph $G(\HH)$ (as in the proof of~\cite[Theorem 7.3]{dallard2022firstpaper}), which has $\mathcal{O}(n^r)$ vertices, and if this graph is $W_4$-induced-minor-free, \cref{MWIS for W4-im-free graphs} applies and leads to an $\mathcal{O}(n^{3r})$ algorithm for the \textsc{Max Weight Independent $\FF$-Packing} problem on $G$.}
Unfortunately, this statement is in general not true, as the following example shows.
Let $\FF = \{K_2\}$ and let $G$ be the $6$-cycle.
Then $G$ is $W_4$-induced-minor-free.
However, the graph $G(\HH)$ obtained with $\HH = \HH(G,\FF)$, that is, the square of the line graph of $G$, has the property that deleting any one vertex from it results in a graph isomorphic to $W_4$.\hfill$\blacktriangle$
\end{remark}

\section{$K_5^-$-induced-minor-free graphs}\label{K5-induced-minor-free}

We now apply the approach outlined at the beginning of \cref{W4-induced-minor-free} to the case of $K_5^-$-induced-minor-free graphs.
As before, we start with characterizing the $3$-connected graphs in the class.
Recall that a \emph{wheel} is any graph $W_n$ obtained from a cycle $C_n$, $n\ge 4$, by adding to it a universal vertex.
Our proof makes use of the following classical result on $3$-connected graphs.

\begin{theorem}[Tutte~\cite{MR0140094}]\label{3conn-contraction}
Every $3$-connected graph with at least $5$ vertices has an edge whose contraction results in a $3$-connected graph.
\end{theorem}

\begin{theorem}\label{3-connected K_5^-induced-minor-free}
For every graph $G$, the following statements are equivalent.
\begin{enumerate}
\item $G$ is $3$-connected and $K_{5}^-$-induced-minor-free.
\item $G$ is either a complete graph with at least four vertices, a wheel, $K_{3,3}$ or $\overline{C_6}$.
\end{enumerate}
\end{theorem}

\begin{proof}
It is straightforward to verify that if $G$ is either a complete graph $K_n$ with $n\ge 4$, a wheel, $G\cong K_{3,3}$, or $G\cong \overline{C_6}$, then
$G$ is $3$-connected and $K_{5}^-$-induced-minor-free.

We prove the converse direction using induction on $n = |V(G)|$.
So let $G$ be a $3$-connected $K_{5}^-$-induced-minor-free graph.
Since $G$ is $3$-connected, $n\ge 4$, and if $n = 4$, then $G$ is complete since otherwise it would contain a vertex of degree at most two.
Suppose that $n \ge 5$. By \cref{3conn-contraction}, $G$ has an edge $e = uv$ whose contraction results in a $3$-connected graph $G'$. Since
$G'$ is also $K_{5}^-$-induced-minor-free, the induction hypothesis implies that $G'$ is either a complete graph, a wheel, $G'\cong K_{3,3}$, or $G'\cong \overline{C_6}$. We analyze each of the four cases separately.

\medskip
\noindent \emph{Case 1: $G'$ is a complete graph.}\\
In this case, $N_G(u)\cup N_G(v) = V(G)\setminus\{u,v\}$.
Let $a = |N_G(u)\setminus N_G(v)|$,
$b = |N_G(u)\cap N_G(v)|$, and
$c = |N_G(v)\setminus N_G(u)|$.
If $a = c = 0$, then $G$ is complete.
We may thus assume by symmetry that $a>0$.
If $b\ge 2$, then any two vertices from $N_G(u)\cap N_G(v)$, along with $u$, $v$, and a vertex from $N_G(u)\setminus N_G(v)$, would induce a subgraph of $G$ isomorphic to $K_5^-$, a contradiction. Thus $b\le 1$.
Since $G$ is $3$-connected, we have $d_G(v)\ge 3$ and hence $c = d_G(v)-(b+1)\ge 2-b>0$.
If $a = c = 1$, then the $3$-connectedness of $G$ implies that $b = 1$, and $G$ is isomorphic to $W_4$. We may thus assume by symmetry that $a\ge 2$.
If $b = 1$, then any two vertices from $N_G(u)\setminus N_G(v)$, along with $u$, a vertex from $N_G(u)\cap N_G(v)$, and a vertex from $N_G(v)\setminus N_G(u)$, form an induced subgraph of $G$ isomorphic to $K_5^-$, a contradiction.
Thus $b = 0$ and consequently $c = d_G(v)-1\ge 2$.
But now, the graph obtained from the subgraph $H$ of $G$ induced by $u$, $v$, any two vertices from $N_G(u)\setminus N_G(v)$, and any two vertices from $N_G(v)\setminus N_G(u)$ by contracting an edge from $u$ to one of the vertices in $V(H)\cap (N_G(u)\setminus N_G(v))$ is isomorphic to $K_5^-$, a contradiction.

\medskip
\noindent \emph{Case 2: $G'$ is a wheel $W_{n-2}$ with $n\ge 6$.}\\
Let $x$ be the universal vertex in $G'$ and let $C$ be the cycle $G'-x$, with a cyclic order of the vertices $v_1,\ldots, v_{n-2}$.
We analyze two subcases depending on whether the vertex $w$ of $G'$ to which the edge $uv$ was contracted corresponds to $x$,
the central vertex of the wheel, or not.

\medskip
\noindent \emph{Case 2.1: $w = x$.}\\
In this case, $V(C)\subseteq N_G(u)\cup N_G(v)$.
Since $G$ is $3$-connected, each of $u$ and $v$ has at least two neighbors on $C$.
Suppose first that each of $u$ and $v$ has only two neighbors on $C$.
Since $C$ has at least four vertices, the neighborhoods of $u$ and $v$ on $C$ are disjoint.
Thus $|V(C)| = 4$ and $G$ is either $\overline{C_6}$ or $K_{3,3}$, depending on whether the two neighbors of $u$ on $C$ are adjacent or not.
We may thus assume that one of $u$ and $v$, say $v$, has at least three neighbors on $C$.
If $u$ and $v$ are both adjacent to three consecutive vertices on $C$, say
$v_j,v_{j+1},v_{j+2}$ (indices modulo $n-2$), then the subgraph of $G$ induced by $\{u,v,v_j,v_{j+1},v_{j+2}\}$ is isomorphic to $K_5^-$, a contradiction.
Thus, we may assume in particular that $u$ is not adjacent to $v_1$, and, consequently, $v$ is adjacent to $v_1$. Let $v_{i}$ and $v_j$ be the two neighbors of $u$ on $C$ such that $i$ is as small as possible and $j$ is as large as possible. Then $1<i<j\le n-2$.
Now, if $v$ has a neighbor in $\{v_2,\ldots, v_{j-1}\}$ and a neighbor in $\{v_{j},\ldots, v_{n-2}\}$, then contracting in $G$ all the edges of the paths $(v_2,v_3,\ldots, v_{j-1})$ and $(v_{j},v_{j+1},\ldots, v_{n-2})$ results in a graph isomorphic to $K_5^-$, a contradiction.
Similarly, if $v$ has a neighbor in $\{v_2,\ldots, v_{i}\}$ and a neighbor in $\{v_{i+1},\ldots, v_{n-2}\}$, then contracting in $G$ all the edges of the paths $(v_2,v_3,\ldots, v_{i})$ and $(v_{i+1},v_{i+2},\ldots, v_{n-2})$ results in a graph isomorphic to $K_5^-$, a contradiction.
Next, if $v$ has at least two neighbors in $\{v_i,\ldots, v_j\}$, say $v_{k_1}$ and $v_{k_2}$, with $k_1<k_2$, then contracting in $G$ all the edges of the paths $(v_2,v_3,\ldots, v_{k_1})$ and $(v_{k_1+1},v_{k_2+2},\ldots, v_{n-2})$ results in a graph isomorphic to $K_5^-$, a contradiction.
Thus, all the neighbors of $v$ in $\{v_2,\ldots, v_{n-2}\}$ are in $\{v_2,\ldots, v_{i-1}\}$ or in $\{v_{j+1},\ldots, v_{n-2}\}$.
We may assume without loss of generality that all the neighbors of $v$ are in $\{v_2,\ldots, v_{i-1}\}$.
Consequently, $i>3$ and the vertices $v_2$ and $v_3$ are adjacent to $v$ and non-adjacent to $u$.
It follows that contracting in $G$ all the edges of the paths $(v_3,v_4,\ldots, v_{j-1})$ and $(v_{j},v_{j+1},\ldots, v_{n-2},v)$ results in a graph isomorphic to $K_5^-$, a contradiction.

\medskip
\noindent \emph{Case 2.2: $w \neq x$.}\\
We may assume without loss of generality that $w = v_1$. Thus, $N_{G}(u)\cup N_G(v) = \{u,v,v_{n-2},x,v_2\}$.
Since $G$ is $3$-connected, each of the vertices $v_2$ and $v_{n-2}$ must have a neighbor in the set $\{u,v\}$.
If the edges in $G$ having one endpoint in $\{v_2,v_{n-2}\}$  and the other one in
$\{u,v\}$ form a matching of size two, then
$u$ and $v$ must both be adjacent to $x$ since $G$ is $3$-connected, and hence $G$ is a wheel, $W_{n-1}$, in this case.
We may thus assume without loss of generality that $u$ is adjacent to both $v_2$ and $v_{n-2}$.
Furthermore, since $d_G(v)\ge 3$, we may assume that $v$ is adjacent to $v_{n-2}$.
Suppose that $v$ is not adjacent to $x$. Then $v$ is adjacent to $v_2$ and $u$ is adjacent to $x$, and
contracting in $G$ the edge $vv_2$ and all the edges of the path $(v_3,\ldots, v_{n-3})$ results in a graph isomorphic to $K_5^-$, a contradiction.
Hence, $v$ is adjacent to $x$. But now, contracting in $G$ the edge $uv_2$ and all the edges of the path $(v_3,\ldots, v_{n-3})$ results in a graph isomorphic to $K_5^-$, a contradiction.

\medskip
\noindent \emph{Case 3: $G'$ is isomorphic to $K_{3,3}$.}\\
Let $w$ be the vertex of $G'$ to which the edge $uv$ was contracted and let $A = \{u_1,u_2,w\}$ and $B = \{v_1,v_2,v_3\}$ be the two independent sets partitioning $V(G')$.
Since $A$ forms an independent set in $G'$, it follows that both $u$ and $v$ are non-adjacent with the vertices $u_1$ and $u_2$ in $G$.
On the other hand, $B\subseteq N_G(u)\cup N_G(v)$, and,
since $G$ is $3$-connected, each of $u$ and $v$ has at least two neighbors in $B$.
We may assume without loss of generality that $u$ and $v$ are both adjacent to $v_2$.
Since each vertex in the set $\{u,v\}$ has a neighbor in the set $\{v_1,v_3\}$ and vice versa, we may assume without loss of generality that $u$ is adjacent to $v_1$ and $v$ is adjacent to $v_3$.
It follows that contracting in $G$ the edges $uv_1$ and $vv_3$ results in a graph isomorphic to $K_5^-$, a contradiction.

\medskip
\noindent \emph{Case 4: $G'$ is isomorphic to $\overline{C_6}$.}\\
Let $v_1,\ldots ,v_6$ be a cyclic order of the vertices of the $C_6$.
It follows that in $G'$, the three odd-indexed vertices form a clique, and the same is true for the even-indexed vertices.
Additionally, in $G'$ we also have the edges $v_1v_4$, $v_3v_6$, and $v_2v_5$.
We may assume that $v_1$ is the vertex obtained by contracting the edge $uv$.
Therefore, $N_G(u)\cup N_G(v) = \{u,v,v_3,v_4,v_5\}$ and, since $G$ is $3$-connected, each of the vertices $u$ and $v$ has at least two neighbors among the vertices $\{v_3,v_4,v_5\}$.
Since each vertex in the set $\{u,v\}$ has a neighbor in the set $\{v_3,v_5\}$ and vice versa, we may assume without loss of generality that $u$ is adjacent to $v_3$ and $v$ is adjacent to $v_5$.
Furthermore, we may assume without loss of generality that $u$ is adjacent to $v_4$.
Now, if $v$ is adjacent to $v_3$, then contracting the edges $uv_4$ and $v_2v_5$ results in a graph isomorphic to $K_5^-$, a contradiction.
On the other hand, if $v$ is adjacent to $v_4$, then contracting the edges $vv_5$ and $v_3v_6$ results in a graph isomorphic to $K_5^-$, again a contradiction.
\end{proof}

Using \cref{3-connected K_5^-induced-minor-free} we now derive the following algorithmic result that will allow us to apply \cref{thm:sufficient} to the case of $K_5^-$-induced-minor-free graphs.

\begin{lemma}\label{k5-(104)}
Let $\mathcal{G}$ be the class of $K_5^-$-induced-minor-free graphs.
Then $\mathcal{G}$ is $(1,3)$-tree decomposable and {\sc $(1,3)$-Tree Decomposition($\mathcal{G}$)} can be solved on graphs with $n$ vertices and $m$ edges in time $\mathcal{O}(n+m)$ so that the resulting tree decomposition has at most $n-3$ nodes.
\end{lemma}

\begin{proof}
Let $G$ be a $3$-connected $K_5^-$-induced-minor-free graph, let $n = |V(G)|$, and let $F^*\subseteq F\subseteq E(G)$ such that $F^*$ is $\mathcal{G}$-safe.
We want to compute an $F^*\!$-covering $3$-refined tree decomposition $\widehat{\mathcal{T}}$ of $G$ with residual independence number at most $1$, together with an $F$-mapping of $\widehat{\mathcal{T}}$.

By \cref{3-connected K_5^-induced-minor-free}, $G$ is either a complete graph with at least four vertices, a wheel, $K_{3,3}$, or $\overline{C_6}$.
In constant time we check if $G\cong K_{3,3}$ or $G\cong \overline{C_6}$.
If this is not the case, then $G$ is either a complete graph or a wheel.
One can distinguish among these two cases in constant time using vertex degrees.
We consider each of the four cases independently.

\medskip
\noindent \emph{Case 1: $G$ is a complete graph.}\\
In this case, $G$ admits a trivial $0$-refined tree decomposition $\widehat{\mathcal{T}}$ with a tree $T$ consisting of only one node $t$ and a unique bag $X_t = V(G)$.
The corresponding $F$-mapping of $\widehat{\mathcal{T}}$ maps each edge $e\in F$ to $t$.
Note that $|X_t|\ge 4$ since $G$ is $3$-connected.
We set $U_t$ to be any subset of $X_t$ of cardinality $3$ and claim that $U_t$ is an $F^*$-cover of $G[X_t] = G$.
Suppose this is not the case and let $e= xy$ be an edge in $F^*$ such that $\{x,y\}\cap U_t=\emptyset$.
Then the subgraph of $G-e$ induced by $U_t\cup \{x,y\}$ is isomorphic to $K_5^-$ and thus $G-e\not\in \mathcal{G}$.
However, this contradicts the assumption that $F^*$ is $\mathcal{G}$-safe.
Hence, $\widehat{\mathcal{T}}$ is an $F^*\!$-covering $3$-refined tree decomposition of $G$ with residual independence number at most $1$.

\medskip
\noindent \emph{Case 2: $G$ is a wheel $W_{n-1}$ with $n\ge 5$.}\\
In linear time we identify the universal vertex $v_0$ in $G$ and compute a cyclic order $v_1,\ldots, v_{n-1}$ of the vertices of the cycle $G-v_0$.
We construct the desired $3$-refined tree decomposition $\widehat{\mathcal{T}} = (T,\{(X_t,U_t)\}_{t\in V(T)})$ of $G$ as follows.
The tree $T$ is an $(n-3)$-vertex path $(t_1,\ldots, t_{n-3})$; for each $i\in \{1,\ldots, n-3\}$, the bag $X_{t_i}$ consists of the vertices $\{v_0,v_i,v_{i+1},v_{n-1}\}$, and we set $U_{t_i}$ to an arbitrary subset of $X_{t_i}$ with cardinality $3$.
Following the fact that $|X_t \setminus U_t| \leq 1$, we get that $\widehat{\mathcal{T}}$ is an $F^*\!$-covering $3$-refined tree decomposition of $G$ with residual independence number at most $1$.
The corresponding $F$-mapping of $\widehat{\mathcal{T}}$ can be obtained as follows.
For each edge $e = v_iv_j\in F$ with $0\le i<j\le n-1$, we map $e$ to $t_e\in V(T)$ where
\begin{align*}
t_e = \left\{\begin{array}{ll}
t_{i}               & \text{if~} 1\le i <j \le n - 2 \text{~or~} (i,j) = (1,n-1)\,,\\
t_j   & \text{if~} i = 0 \text{~and~} j\le n-3\,,\\
t_{n-3}   & \text{otherwise}\,.\\
\end{array}\right.
\end{align*}
Note that both $\widehat{\mathcal{T}}$ and the $F$-mapping can be obtained in linear time.

\medskip
\noindent \emph{Case 3: $G$ is $K_{3,3}$.}\\
Let $A$ and $B$ be the two independent sets partitioning $V(G)$.
Then $G$ has a tree decomposition such that every bag $X_t$ contains all the vertices of $A$ and exactly one vertex of $B$, and thus $|X_t|\le 4$.
This tree decomposition can be turned into an $F^*\!$-covering $3$-refined tree decomposition $\widehat{\mathcal{T}}$ with residual independence number at most one by defining, for instance, $U_t = A$ for all nodes $t$.
In this case, both the $3$-refined tree decomposition $\widehat{\mathcal{T}}$ and an arbitrary $F$-mapping of $\widehat{\mathcal{T}}$ can be obtained in constant time.

\medskip
\noindent \emph{Case 4: $G$ is $\overline{C_6}$.}\\
Let $u$ and $v$ be two non-adjacent vertices of $G$ and let
$(u_1,u_2,u_3,u_4)$ be a $4$-vertex path formed by the vertices of $G$ other than $u$ and $v$.
Labeling the vertices of a $3$-vertex path with the sets $\{u_1,u_2,u,v\}$, $\{u_2,u_3,u,v\}$, $\{u_3,u_4,u,v\}$ in order yields a tree decomposition of $G$ with bags of size $4$.
Again, the tree decomposition can be turned into an $F^*\!$-covering $3$-refined tree decomposition $\widehat{\mathcal{T}}$ with residual independence number at most one by taking, for each bag $X_t$, the set $U_t$ to be an arbitrary subset of $X_t$ with cardinality $3$.
This $3$-refined tree decomposition $\widehat{\mathcal{T}}$ and an arbitrary $F$-mapping of $\widehat{\mathcal{T}}$ can be obtained in constant time.

In each case, the computed $3$-refined tree decomposition has at most $n-3$ nodes.
\end{proof}

\Cref{thm:sufficient,k5-(104)} imply the following.

\begin{theorem}\label{K_5-tree-independence}
For any $K_5^-$-induced-minor-free $n$-vertex graph $G$, one can compute in linear time a $3$-refined tree decomposition of $G$ with $\mathcal{O}(n)$ nodes and residual independence number at most $1$.
\end{theorem}

\begin{proof}
The class of $K_5^-$-induced-minor-free graphs is closed under induced topological minors, and \cref{k5-(104)} shows that \cref{thm:sufficient} is satisfied with $k = 1$, $\ell = 3$, $f(n,m) = \mathcal{O}(n+m)$, and $g(n,m) = n$.
Therefore, given a $K_5^-$-induced-minor-free graph $G$ with $n$ vertices and $m$ edges,
\cref{thm:sufficient} implies that one can compute a $3$-refined tree decomposition of $G$ with $\mathcal{O}(n)$ nodes and residual independence number at most $\max \{3-\ell,k\}= 1$ in time $\mathcal{O}(n+m+f(4n,3m))$, which is $\mathcal{O}(n+m)$.
\end{proof}

\begin{remark}\label{K5-robust}
Similarly as in \Cref{W4-robust}, one can turn the algorithm for computing a $3$-refined tree decomposition of a $K_5^-$-induced-minor-free graph into a polynomial-time robust algorithm, as follows.
First, after we compute the SPQR tree $S$ of a given graph $G$, we check for every $R$-node $a$ in $S$ if the corresponding graph $G_a$ is a complete graph, a wheel, a $K_{3,3}$, or a $\overline{C_6}$.
If $G_a$ has at most six vertices, we check in constant time if $G_a$ corresponds to any of the four cases, and if not, the algorithm returns that the graph is not $K_5^-$-induced-minor-free.
If $G_a$ has more than six vertices, then we check in linear time whether $G_a$ is a complete graph or a wheel.
If this is also not the case, then the algorithm returns that the graph is not $K_5^-$-induced-minor-free.
Assume now that each graph $G_a$ corresponding to an $R$ node $a$ of $S$ is a complete graph, a wheel, a $K_{3,3}$, or a $\overline{C_6}$.
Then, once we compute a $3$-refined tree decomposition $\widehat{\mathcal{T}}$ of $G_a$ as in the proof of \cref{k5-(104)}, we check for each bag $X$ of $\widehat{\mathcal{T}}$ if the graph $G[X]$ contains an independent set of size $5$.
(Note that this can only happen in the case when $G_a$ is a clique.)
If some graph $G[X]$ contains an independent set of size $5$, then the algorithm returns that $G$ is not $K_5^-$-induced-minor-free;
otherwise, the algorithm correctly computes a $3$-refined tree decomposition of $G$ with residual independence number at most $1$.
\hfill$\blacktriangle$
\end{remark}

Regarding bounding the treewidth in terms of the clique or the Hadwiger numbers, the situation for the class of $K_5^-$-induced-minor-free graphs is similar as for the class of $W_4$-induced-minor-free graphs.
\Cref{bounded tin implies bounded tw-omega-refined,K_5-tree-independence} imply the existence of a linear $(\tw,\omega)$-binding function $f(p) = p+2$ for the class of $K_5^-$-induced-minor-free graphs, for which no polynomial binding function was known prior to this work.
Furthermore, our approach leading to the proof of \cref{K_5-tree-independence} can be used to show the following.

\begin{proposition}\label{Hadwiger-number-K_5-induced-minor-free-graphs}
If $G$ is a $K_5^-$-induced-minor-free graph, then $\tw(G) = \eta(G)-1$.
\end{proposition}

\cref{K_5-tree-independence,observation-ell-tree-alpha} imply the following.

\begin{corollary}\label{K_5-tree-independence-number}
The tree-independence number of any $K_5^-$-induced-minor-free graph is at most $4$.
\end{corollary}

\begin{remark}
Consider again the family $\{F_q\}_{q\ge 4}$ of graphs from \cref{remark:sharp}.
The same arguments as used therein to show that graphs $F_q$ are $W_4$-induced-minor-free also show that these graphs are $K_5^-$-induced-minor-free.
Thus, the bound on the tree-independence number given by \cref{K_5-tree-independence-number} is sharp; there exist arbitrarily large $2$-connected $K_5^-$-induced minor-free graphs with tree-independence number~$4$.\hfill$\blacktriangle$
\end{remark}

\Cref{K_5-tree-independence,thm:bounded-ell-refined-tree-independence-number} have the following algorithmic consequence for the \textsc{Max Weight Independent Set} problem in the class of $K_5^-$-induced-minor-free graphs.

\begin{corollary}\label{MWIS for K5minus-im-free graphs}
The \textsc{Max Weight Independent Set} problem can be solved in time $\mathcal{O}(n^{3})$ for $n$-vertex $K_5^-$-induced-minor-free graphs.
\end{corollary}

\begin{remark}\label{refined-improvement-2}
Similarly as for the class of $W_4$-induced-minor-free graphs (see \cref{refined-improvement-1}), we could apply \cref{thm:bounded-tree-independence-number} instead of \cref{thm:bounded-ell-refined-tree-independence-number}, using usual tree decompositions instead of $\ell$-refined ones.
However, such an approach would again only lead to a running time of $\mathcal{O}(n^{6})$ instead of $\mathcal{O}(n^{3})$.
\hfill$\blacktriangle$
\end{remark}

\medskip
\cref{K_5-tree-independence,thm:bounded-tree-independence-number-packings} imply the existence of a polynomial-time algorithm for the \textsc{Max Weight Independent Packing} problems in the class of $K_5^-$-induced-minor-free graphs.

\begin{theorem}\label{MWIHP for K5minus-im-free graphs}
Given a $K_5^-$-induced-minor-free graph $G$ and a finite family $\mathcal{H} = \{H_j\}_{j \in J}$ of connected nonnull subgraphs of $G$, the \textsc{Max Weight Independent Packing} problem can be solved in time \hbox{$\mathcal{O}(|J|\cdot |V(G)|\cdot (|V(G)| + |J|^4))$}.
\end{theorem}

\begin{proof}
Let $G$ be an $n$-vertex $K_5^-$-induced-minor-free graph, given along with a finite family $\HH = \{H_j\}_{j \in J}$ of connected nonnull subgraphs and a weight function \hbox{$w:J\to \mathbb{Q}_+$}.
By \cref{K_5-tree-independence,observation-ell-tree-alpha}, we can compute in linear time a tree decomposition \hbox{$\mathcal{T} = (T, \{\Bag_t\}_{t\in V(T)})$} of $G$ with $\mathcal{O}(n)$ nodes and independence number at most $4$.
The conclusion now follows from \cref{thm:bounded-tree-independence-number-packings}.
\end{proof}

\begin{remark}
Similarly as in \cref{remark:cant-apply-MWIS-directly}, we can ask if, given a family $\FF$ of connected nonnull graphs, the graph $G(\HH)$ for $\HH = \HH(G,\FF)$ is $K_5^-$-induced-minor-free whenever $G$ is.
Again, this is in general not true, as the following example shows.
Let $\FF = \{K_2\}$ and let $G$ be the graph obtained from the $5$-vertex path $(v_1,\ldots, v_5)$ by adding the edge $v_3v_5$.
Then $G$ is $K_5^-$-induced-minor-free, while the graph $G(\HH)$ is isomorphic to $K_5^-$.
Note that the example from \cref{remark:cant-apply-MWIS-directly} would also work: if $G$ is the $6$-cycle and $\FF = \{K_2\}$, then the graph $G(\HH)$ has the property that contracting any of its edges results in a graph isomorphic to $K_5^-$.\hfill$\blacktriangle$
\end{remark}

\section{The dichotomies}\label{sec:dichotomies}

In order to state the next theorem we require two more definitions.
A graph is \emph{subcubic} if every vertex of it has degree at most three.
We denote by $\mathcal{S}$ the class of forests in which every connected component has at most three leaves (vertices of degree one).

\begin{theorem}[Dallard, Milani{\v{c}}, and {\v{S}}torgel~\cite{DMS-WG2020,dallard2021treewidth}]\label{dichotomy tw-omega}
    For every graph $H$, the following statements hold.
    \begin{enumerate}
        \item The class of $H$-subgraph-free graphs is $(\tw,\omega)$-bounded if and only if $H \in \mathcal{S}$.

        \item The class of $H$-topological-minor-free graphs is $(\tw,\omega)$-bounded if and only if $H$ is subcubic and planar.

        \item The class of $H$-minor-free graphs is $(\tw,\omega)$-bounded if and only if $H$ is planar.
        \item The class of $H$-free graphs is $(\tw,\omega)$-bounded if and only if $H$ is either an induced subgraph of $P_3$ or an edgeless graph.

        \item The class of $H$-induced-topological-minor-free graphs is $(\tw,\omega)$-bounded if and only if $H$ is either an induced topological minor of $C_4$ or $K_4^-$, or $H$ is edgeless.

        \item The class of $H$-induced-minor-free graphs is $(\tw,\omega)$-bounded if and only if $H$ is an induced minor of $W_4$, $K_5^-$, or $K_{2,q}$ for some $q \in \mathbb{Z}_+$, or $H$ is edgeless.
    \end{enumerate}
\end{theorem}

\begin{remark}\label{tw-omega bounded treewidth non-induced}
A graph class closed under edge deletion has bounded tree-independence number if and only if it has bounded treewidth.
This follows from Robertson and Seymour's Grid-Minor Theorem (see~\cite{MR0854606}) stating that graphs with sufficiently large treewidth contain a large (elementary) wall as a minor.
In particular, such graphs contain a subdivision of a large wall as a subgraph, which implies $(\tw,\omega)$-unboundedness and, thus, unbounded tree-independence number.
Hence, for graph classes closed under subgraph, topological minor, or minor relation, bounded tree-independence number is equivalent to bounded treewidth.\hfill$\blacktriangle$
\end{remark}

We can now refine the result of \cref{dichotomy tw-omega} and prove the main result of this section.

\begin{theorem}\label{dichotomy tin}
   For every graph $H$, the following statements hold.
    \begin{enumerate}
        \item The class of $H$-subgraph-free graphs has bounded tree-independence number if and only if $H \in \mathcal{S}$.

        \item The class of $H$-topological-minor-free graphs has bounded tree-independence number if and only if $H$ is subcubic and planar.

        \item The class of $H$-minor-free graphs has bounded tree-independence number if and only if $H$ is planar.
        \item The class of $H$-free graphs has bounded tree-independence number if and only if $H$ is either an induced subgraph of $P_3$ or an edgeless graph.

        \item The class of $H$-induced-topological-minor-free graphs has bounded tree-independence number if and only if $H$ is either an induced topological minor of $C_4$ or $K_4^-$, or $H$ is edgeless.

        \item The class of $H$-induced-minor-free graphs has bounded tree-independence number if and only if $H$ is an induced minor of $W_4$, $K_5^-$, or $K_{2,q}$ for some $q \in \mathbb{Z}_+$, or $H$ is edgeless.
    \end{enumerate}
\end{theorem}

\begin{proof}
  Fix a graph $H$ and one of the six considered graph containment relations, and let $\mathcal{G}$ be the class of graphs excluding $H$ with respect to this relation.
  Assume first that $\mathcal{G}$ has bounded tree-independence number.
Then, by \cref{bounded tin implies bounded tw-omega-refined}, $\mathcal{G}$ is $(\tw,\omega)$-bounded.
Thus, for each graph containment relation, the forward implication holds by \cref{dichotomy tw-omega}.

    Consider first the subgraph, topological minor, and minor relations.
    Assume that $\mathcal G$ is $(\tw,\omega)$-bounded.
    Then $\mathcal G$ has bounded treewidth by \cref{tw-omega bounded treewidth non-induced}, and therefore also bounded tree-independence number.

    We may thus assume that the considered relation is the induced subgraph, induced topological minor, or induced minor relation.
    Clearly, if $H$ is edgeless, then every graph $G \in \mathcal G$ has independence number at most $|V(H)|$, which implies that $G$ has tree-independence number at most $|V(H)|$.

    Consider now the induced subgraph relation and assume that $H$ is an induced subgraph of $P_3$.
    Then every graph $G$ in $\mathcal G$ is $P_3$-free, and thus every connected component of $G$ is a complete graph.
    Hence, there clearly exists a tree decomposition of $G$ with independence number $1$.
    So the theorem holds for the induced subgraph relation.

    Consider next the induced topological minor relation.
    Assume first that $H$ is an induced topological minor of $C_4$.
    Then $\mathcal{G}$ is a subclass of the class of $C_4$-induced-topological-minor-free graphs, which is precisely the class of chordal graphs.
    By \cref{chordal}, every chordal graph has tree-independence number at most $1$.
    Assume now that $H$ is an induced topological minor of $K_4^-$.
    Then $\mathcal{G}$ is a subclass of the class of block-cactus graphs~\cite{DMS-WG2020,dallard2021treewidth}.
    By \cref{cor:block-cactus-graphs}, every graph in $\mathcal{G}$ has tree-independence number at most $2$.
    Hence, the theorem holds for the induced topological minor relation.

    Now, consider the induced minor relation.
    If $H$ is an induced minor of $W_4$ or $K_5^-$, then every graph in $\mathcal G$ is $W_4$-induced-minor-free or $K_5^-$-induced-minor-free, and so \cref{W_4-tree-independence-number,K_5-tree-independence-number} imply an upper bound of $4$ on the tree-independence number of graphs in $\mathcal G$.
    Finally, if $H$ is an induced minor of $K_{2,q}$ for some $q\ge 2$, then every graph in $\mathcal G$ is $K_{2,q}$-induced-minor free.
    As shown in \cref{thm:tree-independence-number-K_2q-induced-minor-free-graphs}, such graphs have tree-independence number at most $2q-2$.
    Hence, the theorem holds for the induced minor relation.
\end{proof}

\Cref{bounded tin implies bounded tw-omega-refined,dichotomy tin,dichotomy tw-omega}, and the proof of \cref{dichotomy tin}, imply the following.

\begin{corollary}\label{equivalencies without H}
For every graph $H$ and each of the six graph containment relations (the subgraph, topological minor, and minor relations, and their induced variants), the following statements are equivalent for the class $\G$ of graphs excluding $H$ with respect to the relation.
\begin{enumerate}
\item $\G$ is $(\tw,\omega)$-bounded.
\item $\G$ is polynomially $(\tw,\omega)$-bounded.
\item $\G$ has bounded tree-independence number.
\end{enumerate}
Furthermore, whenever the above conditions are satisfied, there is a polynomial-time algorithm for computing a tree decomposition with bounded independence number of a graph in $\G$.
\end{corollary}

\section{Concluding remarks}\label{sec:open-questions}

In conclusion, we summarize the main questions left open by this work and the two preceding works of the series~\cite{dallard2021treewidth,dallard2022firstpaper}, and discuss their interrelations.

As mentioned in \cref{sec:K2q-induced-minor-free}, we do not know if the bound on the tree-independence number of $K_{2,q}$-induced-minor-free graphs given by \cref{thm:tree-independence-number-K_2q-induced-minor-free-graphs} is sharp.

\begin{question}
For an integer $q\ge 3$, what is the maximum tree-independence number over all $K_{2,q}$-induced-minor-free graphs?
\end{question}

As explained in \cref{sec:K2q-induced-minor-free}, the answer is somewhere between $q-1$ and $2(q-1)$.

We say that a graph class $\mathcal{G}$ is \emph{$\alpha$-easy} if the \textsc{Max Weight Independent Set} problem is solvable in polynomial time for graphs in $\mathcal{G}$.
\Cref{theorem-main} gives an infinite family of planar graphs $H$ such that the class of $H$-induced-minor-free graphs is $\alpha$-easy.
In general, however, the problem of classifying planar graphs $H$ with respect to the complexity of the \textsc{Max Weight Independent Set} problem in the class of graphs excluding $H$ as an induced minor is a widely open question, even in the case when $H$ is a path or cycle.

\begin{question}
Is the class of $H$-induced-minor-free graphs $\alpha$-easy for every planar graph $H$?
\end{question}

Recently, Korhonen has shown that for any planar graph $H$, the \textsc{Max Weight Independent Set} problem can be solved in subexponential time in the class of $H$-induced-minor-free graphs~\cite{Korhonen2022}.

\medskip
The remaining questions can be phrased in terms of various properties of graph classes, as summarized in \cref{fig questions}.

\begin{figure}
    \centering
    \begin{tikzpicture}[yscale=1.25]
    \usetikzlibrary{arrows.meta}
    \usetikzlibrary{calc}
        \tikzset{every node/.style={rectangle, draw, very thick, inner sep = 5pt, outer sep = 2pt, align = center, rounded corners=2.pt}}

        \node[text width = 6cm, align = center] (alpha-easy) at (0,0) {\strut $\alpha$-easy};
        \node[text width = 6cm, align = center, below = of alpha-easy] (tw-w-B) {\strut $(\tw,\omega)$-bounded};
        \node[text width = 6cm, align = center, below = of tw-w-B] (poly-tw-w-B) {\strut polynomially $(\tw,\omega)$-bounded};
        \node[text width = 6cm, align = center, below = of poly-tw-w-B] (B-tree-alpha) {\strut bounded tree-independence number};
        \node[text width = 6cm, align = center, below = of B-tree-alpha] (EW-tree-alpha) {\strut bounded and efficiently witnessed tree-independence number};

        \draw [-{Latex[round,length=2.5mm,width=2.5mm]},very thick] ($(EW-tree-alpha.north)+(-0.3,0)$) -- ($(B-tree-alpha.south)+(-0.3,0)$);
        \draw [-{Latex[round,length=2.5mm,width=2.5mm]},very thick] ($(B-tree-alpha.north)+(-0.3,0)$) -- ($(poly-tw-w-B.south)+(-0.3,0)$);
        \draw [-{Latex[round,length=2.5mm,width=2.5mm]},very thick] ($(poly-tw-w-B.north)+(-0.3,0)$) -- ($(tw-w-B.south)+(-0.3,0)$);
        \draw [-{Latex[round,length=2.5mm,width=2.5mm]}, very thick, dashed] ($(tw-w-B.north)+(-0.3,0)$) -- node[auto,swap,fill=none,draw=none]{\strut \Cref{question3}} ($(alpha-easy.south)+(-0.3,0)$);

        \draw [-{Latex[round,length=2.5mm,width=2.5mm]}, very thick, dashed] ($(B-tree-alpha.south)+(0.3,0)$) -- node[auto,fill=none,draw=none]{\strut \Cref{computing-tree-decompositions}} ($(EW-tree-alpha.north)+(0.3,0)$);
        \draw [-{Latex[round,length=2.5mm,width=2.5mm]}, very thick, dashed] ($(poly-tw-w-B.south)+(0.3,0)$) -- ($(B-tree-alpha.north)+(0.3,0)$);
        \draw [-{Latex[round,length=2.5mm,width=2.5mm]}, very thick, dashed] ($(tw-w-B.south)+(0.3,0)$) -- node[auto,fill=none,draw=none]{\strut \Cref{q1}} ($(poly-tw-w-B.north)+(0.3,0)$);

        \draw [-{Latex[round,length=2.5mm,width=2.5mm]},very thick] (EW-tree-alpha.north west) to[bend left = 65] ($(alpha-easy.west)+(0,-0.1)$);

        \draw [-{Latex[round,length=2.5mm,width=2.5mm]}, dashed,very thick] (B-tree-alpha.north west) to[bend left = 55] node[auto,pos=0.15,right=0pt,fill=none,draw=none] {\strut \Cref{mwis-complexity-for-bounded-tin}} (alpha-easy.south west);

        \draw [-{Latex[round,length=2.5mm,width=2.5mm]}, dashed,very thick] (poly-tw-w-B.north west) to[bend left = 35] ($(alpha-easy.south west)+(0.3,0)$);

        \draw [-{Latex[round,length=2.5mm,width=2.5mm]}, very thick, dashed] (tw-w-B.east) to[bend left = 75] node[near start,right,fill=none,draw=none]{\strut \Cref{main-conjecture}} (B-tree-alpha.east);

        \draw [-{Latex[round,length=2.5mm,width=2.5mm]}, very thick, dashed] (tw-w-B.south east) to[bend left = 55] (EW-tree-alpha.east);

        \draw [-{Latex[round,length=2.5mm,width=2.5mm]}, very thick, dashed] (poly-tw-w-B.south east) to[bend left = 35] (EW-tree-alpha.north east);
    \end{tikzpicture}
    \caption{Various properties of graph classes and their interrelations. A full line arrow from property $A$ to property $B$ means that every graph class having property $A$ also has property $B$.
    Dashed lines correspond to open questions.}
    \label{fig questions}
\end{figure}
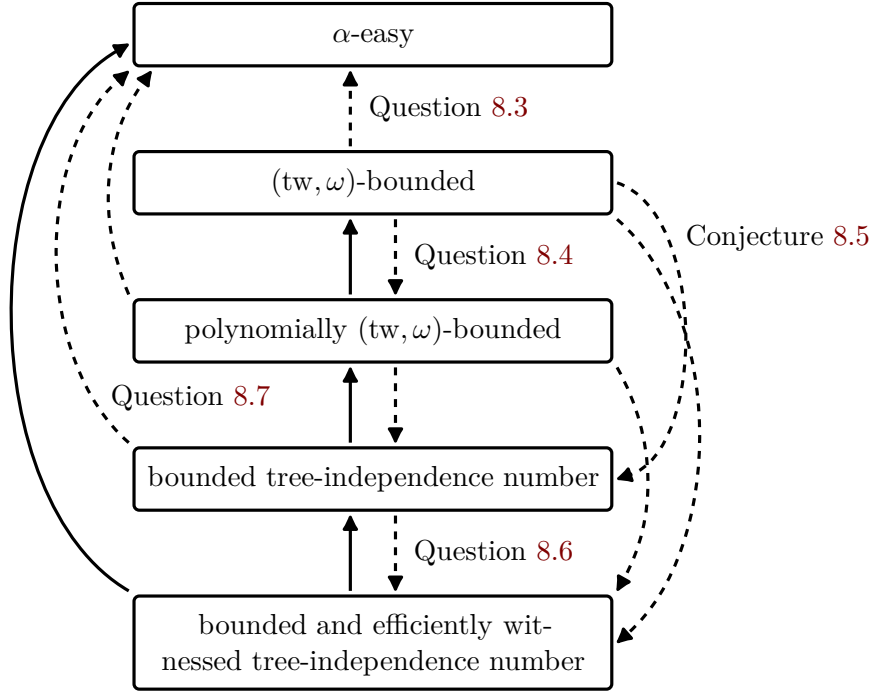

To explain the nontrivial implications in the diagram given in \cref{fig questions}, note that:
\begin{itemize}
\item Corollary 5.3 in \cite{dallard2022firstpaper} (cf.~\cref{thm:bounded-tree-independence-number}) implies that every graph class $\mathcal{G}$ with bounded and efficiently witnessed tree-independence number is $\alpha$-easy.
\item Lemma 3.2 in~\cite{dallard2022firstpaper} (cf.~\cref{bounded tin implies bounded tw-omega-refined}) implies that every graph class with bounded tree-independence number is polynomially $(\tw,\omega)$-bounded.
\end{itemize}

By \cref{tin-of-Knn}, for each $n\ge 1$, the tree-independence number of complete bipartite graph $K_{n,n}$ is $n$.
Because the \textsc{Max Weight Independent Set} problem is solvable in linear time in the class of complete bipartite graphs, it follows that not every $\alpha$-easy hereditary graph class is $(\tw,\omega)$-bounded, that is, the topmost possible implication in the diagram given in \cref{fig questions} cannot be an equivalence.
However, we do not know if the four properties at the bottom of the diagram are distinct, leaving open $9$ pairs $(A,B)$ of properties for which we do not know if every graph class with property $A$ also has property $B$ (corresponding to the dashed lines in \cref{fig questions}).
We now explicitly state and discuss some of these questions.

In~\cite{dallard2021treewidth}, we asked whether there is a $(\tw,\omega)$-bounded graph class in which the \textsc{Max Weight Independent Set} problem is \NP-hard.
Unless $\P = \NP$, the question can be equivalently formulated as follows.

\begin{question}\label{question3}
Is every $(\tw,\omega)$-bounded graph class $\alpha$-easy?
\end{question}

As shown in this paper, \Cref{question3} has a positive answer when restricted to the family of classes of graphs excluding a single graph with respect to one of the six considered relations.

In~\cite{DMS-WG2020,dallard2021treewidth}, we also posed the following question.

\begin{question}\label{q1}
Is every $(\tw,\omega)$-bounded graph class polynomially $(\tw,\omega)$-bounded?
\end{question}

We showed~\cite{DMS-WG2020,dallard2021treewidth} that \cref{q1} has an affirmative answer for the $(\tw,\omega)$-bounded graph classes defined by forbidding a single graph $H$ with respect to the minor, the topological minor, the subgraph relation, or their induced variants, except possibly when $H\in \{W_4,K_5^-\}$ is excluded as an induced minor.
\Cref{equivalencies without H} strengthens these results by showing that for every $H$ and each of the six relations, $(\tw,\omega)$-boundedness, polynomial $(\tw,\omega)$-boundedness, bounded tree-independence number, as well as bounded and efficiently witnessed tree-independence number are in fact equivalent concepts.
Motivated by this result, we boldly pose the following conjecture.

\begin{conjecture}\label{main-conjecture}
Let $\mathcal{G}$ be a hereditary graph class.
Then $\mathcal{G}$ is $(\tw,\omega)$-bounded if and only if $\mathcal{G}$ has bounded tree-independence number.
\end{conjecture}

As a first step towards \cref{main-conjecture}, it would be interesting to determine which of the $(\tw,\omega)$-bounded graph classes defined by finitely many forbidden induced subgraphs (see~\cite{MR4385180} for a characterization) have bounded tree-independence number.

In the second paper of the series~\cite{dallard2022firstpaper}, we posed the following two interrelated questions.

\begin{question}\label{computing-tree-decompositions}
Is there a computable function $f:\mathbb{Z}_+\to\mathbb{Z}_+$ such that for every positive integer $k$ there exists a polynomial-time algorithm that takes as input a graph $G$ with tree-independence number at most $k$ and computes a tree decomposition of $G$ with independence number at most $f(k)$?
\end{question}

\begin{question}\label{mwis-complexity-for-bounded-tin}
Is every graph class with bounded tree-independence number $\alpha$-easy?
\end{question}

\medskip
\noindent\textbf{Note.}
After the paper was submitted for publication, we became aware of the work of Yolov~\cite{MR3775804}, who considered two parameters based on tree decompositions, including the tree-independence number (named \emph{$\alpha$-treewidth} and denoted by $\alpha$-$\tw$).
He gave an algorithm that, given an $n$-vertex graph $G$ and an integer $k$, in time $n^{\mathcal{O}(k^3)}$ either computes a tree decomposition of $G$ with independence number $\mathcal{O}(k^3)$ or correctly determines that the tree-independence number of $G$ is larger than $k$.
In particular, this provides an affirmative answer to~\Cref{computing-tree-decompositions} and consequently to \Cref{mwis-complexity-for-bounded-tin}.
Yolov's result was recently improved further by Dallard, Golovach, Fomin, Korhonen, and Milani{\v c}~\cite{dallard2022computing}, who improved the running time to $2^{\mathcal{O}(k^2)} n^{\mathcal{O}(k)}$ and the upper bound on the independence number of the computed tree decomposition to $8k$.

\subsection*{Acknowledgements}

We are grateful to Erik Jan van Leeuwen for telling us about the work of Yolov~\cite{MR3775804}.
This work is supported in part by the Slovenian Research Agency (I0-0035, research programs P1-0285 and P1-0383, research projects J1-3001, J1-3002, J1-3003, J1-4008, J1-4084, J1-9110, J7-3156, N1-0102, N1-0160, and a Young Researchers Grant).

\end{document}